\documentclass[11pt]{amsart}
\usepackage[margin=1in]{geometry}
\usepackage{graphicx} % Required for inserting images
\usepackage{mathtools}
\usepackage{amsthm,amssymb}
\usepackage[usenames,dvipsnames]{xcolor}
\usepackage{tikz,hyperref}
\usetikzlibrary{cd}
\usepackage{enumerate}
\usepackage{combelow}
\usepackage{comment}
\usepackage{youngtab}

\theoremstyle{plain}
\newtheorem{thm}{Theorem}[section]
\newtheorem{prop}[thm]{Proposition}
\newtheorem{lem}[thm]{Lemma}

\newtheorem{cor}[thm]{Corollary}

\newtheorem*{mainthm}{Theorem \ref{thm:IntroSkewing}}
\newtheorem*{combimain}{Theorem \ref{thm: combi main}}

\theoremstyle{remark}
\newtheorem{rmk}[thm]{Remark}

\theoremstyle{definition}
\newtheorem{defn}[thm]{Definition}
\newtheorem{example}[thm]{Example}
\newtheorem{question}[thm]{Question}

\newcommand{\bC}{\mathbb{C}}

\newcommand{\bQ}{\mathbb{Q}}

\newcommand{\inv}{\mathrm{inv}}

\newcommand{\area}{\mathrm{area}}
\newcommand{\dinv}{\mathrm{dinv}}
\newcommand{\Inv}{\mathrm{Inv}}

\newcommand{\arm}{\mathrm{arm}}
\newcommand{\leg}{\mathrm{leg}}
\newcommand{\tdinv}{\mathrm{tdinv}}
\newcommand{\maxtdinv}{\mathrm{maxtdinv}}
\newcommand{\maxInv}{\mathrm{maxInv}}

\newcommand{\hdinv}{\mathrm{hdinv}}
\newcommand{\pathdinv}{\mathrm{pathdinv}}
\newcommand{\WPF}{\mathrm{WPF}}

\newcommand{\LD}{\mathrm{WLD}}
\newcommand{\Stack}{\mathrm{Stack}}
\newcommand{\stack}{\mathrm{stack}}
\newcommand{\sm}{\mathrm{small}}
\newcommand{\bg}{\mathrm{big}}

\newcommand{\pol}{\mathrm{pol}}

\newcommand{\GL}{\mathrm{GL}}

\newcommand{\ch}{\mathrm{ch}}
\newcommand{\Hom}{\mathrm{Hom}}
\newcommand{\Span}{\mathrm{Span}}
\newcommand{\cE}{\mathcal{E}}
\newcommand{\cS}{\mathcal{S}}
\newcommand{\Dyck}{\mathrm{Dyck}}
\newcommand{\sgn}{\mathrm{sgn}}
\newcommand{\talpha}{\widetilde{\alpha}}

\title[A skewing formula for the Rise Delta Theorem]{A combinatorial skewing formula for\\the Rise Delta Theorem}
\author{Maria Gillespie} 
\address{Department of Mathematics, Colorado State University, Fort Collins, CO 80523, USA}
\email{maria.gillespie@colostate.edu}
\thanks{The first author was partially supported by NSF DMS award number 2054391.}

\author{Eugene Gorsky}
\address{Department of Mathematics, University of California Davis\\ One Shields Ave, Davis CA 95616, USA}
\email{egorskiy@math.ucdavis.edu}
\thanks{The second author was partially supported by NSF DMS award number 2302305.}

\author{Sean T. Griffin}
\address{Faculty of Mathematics, University of Vienna, Oskar-Morgenstern-Platz 1, 1090 Vienna, Austria}
\email{sean.griffin@univie.ac.at}
\date{\today}

\begin{document}

\begin{abstract}
    We prove that the symmetric function $\Delta'_{e_{k-1}}e_n$ appearing in the Delta Conjecture can be obtained from the symmetric function in the Rational Shuffle Theorem by applying a Schur skewing operator.  This generalizes a formula by the first and third authors for the Delta Conjecture at $t=0$, and follows from work of Blasiak, Haiman, Morse, Pun, and Seelinger.
    
    Our main result is that we also provide a purely combinatorial proof of this skewing identity, giving a new proof of the Rise Delta Theorem from the Rational Shuffle Theorem.  
\end{abstract}

\maketitle

%\tableofcontents

\section{Introduction}

In the last several decades, many connections have been discovered between Catalan combinatorics and algebraic geometry. One of the most important results in this direction has been Haiman's $(n+1)^{n-1}$ Theorem which asserts that the total dimension of the ring of diagonal coinvariants for the group $S_n$ is $(n+1)^{n-1}$. This formula was proven by Haiman~\cite{Haiman,Haiman2} using deep results on the Hilbert scheme of $n$ points in $\bC^2$. In fact, this ring carries an extra structure as a bigraded $\bC S_n$-module. The bigraded (Frobenius) character of this module has also been shown by Haiman to be equal to $\nabla e_n$, where $e_n$ is an elementary symmetric function and $\nabla$ is a Macdonald eigenoperator \cite{Nabla}. The Shuffle Theorem, which was conjectured in \cite{HHLRU} and proven in \cite{CM} gives a beautiful combinatorial formula \eqref{eq: nabla intro} for $\nabla e_n$ in terms of labeled Dyck paths.

Two prominent generalizations of the Shuffle Theorem are the Rational Shuffle Theorem, which gives a combinatorial formula for $E_{ka,kb}\cdot 1$ where $E_{ka,kb}$ is a certain operator from the elliptic Hall Algebra, and the Delta Conjecture, which gives a combinatorial formula for $\Delta'_{e_{k-1}}e_n$ where $\Delta'_f$ is a class of Macdonald eigenoperators generalizing $\nabla$. In this article, we show that the two are directly related by a Schur skewing operator, as follows.

\begin{thm}\label{thm:IntroSkewing}
    Letting $K = k(n-k+1)$ and $\lambda = (k-1)^{n-k}$, we have
    \begin{equation}\label{eq:IntroSkewingEqn}
    \Delta'_{e_{k-1}}e_n = s_{\lambda}^\perp (E_{K,k}\cdot 1),
    \end{equation}
    where $s_\lambda^\perp$ is the adjoint to multiplication by the Schur function $s_\lambda$. 
\end{thm}

We give two proofs of Theorem \ref{thm:IntroSkewing}; one is algebraic in nature and uses the relation between the Elliptic Hall Algebra and the Shuffle Algebra and certain identities for $E_{K,k}$ studied in  \cite{BHMPS,BHMPS2,Negut}.  The second is a direct combinatorial proof linking the parking functions and statistics from the Rise Delta Conjecture and the Rational Shuffle Theorem, see Theorem~\ref{thm: combi main} and Corollary~\ref{cor: commutative-diagram} below. 

In a forthcoming paper \cite{GGG2}, we give geometric realizations of the Rational Shuffle Theorem (in the integer slope case $(km,k)$, also known as the ``rectangular'' case) and of the Delta Conjecture in terms of affine Springer fibers.  This also makes use of Theorem \ref{thm:IntroSkewing}.

\subsection{Shuffle Theorems}

A \textbf{labeled $(n,n)$ Dyck path} or \textbf{$(n,n)$ word parking function} is a Dyck path in the $n\times n$ grid whose vertical steps are labeled with positive integers, such that the labeling strictly increases up each vertical run. See the left-most example in Figure~\ref{fig:PF-example}. The Shuffle Theorem~\cite{CM} gives the following remarkable combinatorial formula for the evaluation $\nabla e_n$,
\begin{equation}
\label{eq: nabla intro}
\nabla e_n = \sum_{P \in \WPF_{n,n}} t^{\area(P)} q^{\dinv(P)} x^P,
\end{equation}
see Section \ref{sec: background} for relevant definitions. Similarly, given $\gcd(a,b)=1$ one can define a {\bf $(ka,kb)$ rational word parking function} as a lattice path (also knowns as a rational Dyck path) in the $(ka)\times (kb)$ grid that stays weakly above the line $y=ax/b$, starts in the southwest corner $(0,0)$, and ends in the northeast corner $(kb,ka)$, together with a column-strictly-increasing labeling of the up steps (see the middle path in Figure~\ref{fig:PF-example} for an example where $k=3$, $a=2$ and $b=1$). The combinatorial statistics $\area$ and $\dinv$ in the ``combinatorial" right hand side of \eqref{eq: nabla intro} has a natural generalization to rational parking functions.

\begin{figure}
    \centering
    \includegraphics[scale=0.5]{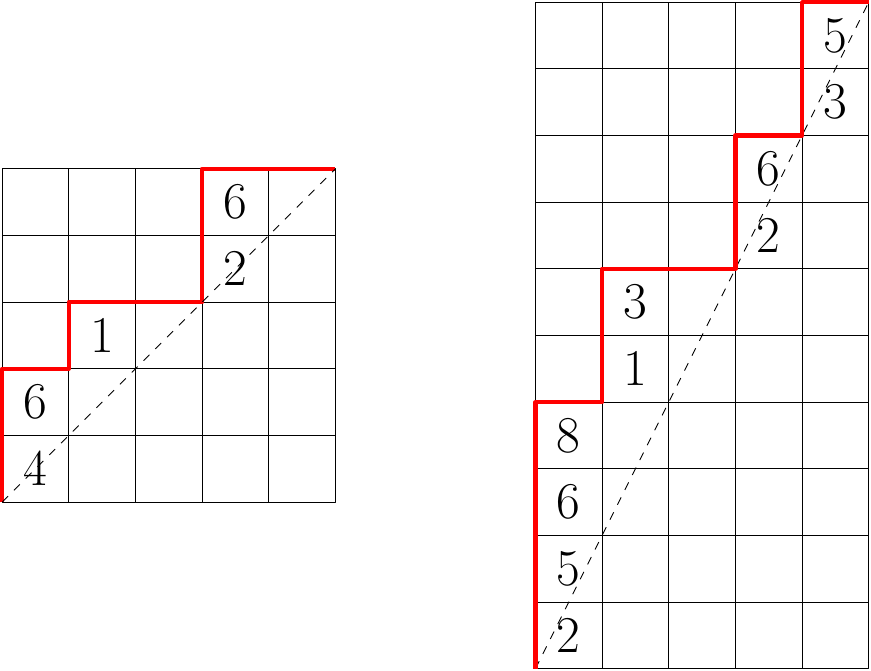}
    \caption{Examples of $(3,3)$ and $(6,3)$ word parking functions}
    \label{fig:PF-example}
\end{figure}

However, to generalize the ``algebraic'' left-hand side one needs to consider the {\bf Elliptic Hall Algebra} $\cE_{q,t}$ defined in \cite{EHA} and extensively studied in the last two decades \cite{RationalShuffle,BHMPS,BHMPS2,GN,Negut,SV2,SV1}. This is a remarkable algebra acting on the space $\Lambda(q,t)$ of symmetric functions in infinitely many variables with coefficients in $\bQ(q,t)$. The Rational Shuffle Theorem, 
proposed by Bergeron, Garsia, Leven and Xin~\cite{RationalShuffle} (see also \cite{GN} for $k=1$ case and a connection to Hilbert schemes), subsequently proven by Mellit~\cite{Mellit}, states that
\[
E_{ka,kb} \cdot (-1)^{k(a+1)} = \sum_{P\in \WPF_{ka,kb}} t^{\area(P)} q^{\dinv(P)} x^P,
\]
where $E_{ka,kb}$ is a particular element of  $\cE_{q,t}$.
Note that our conventions for $a$ and $b$ are flipped from~\cite{RationalShuffle}. 

We will be primarily interested in the case $(ka,kb)=(K,k)$ where $K=k(n-k+1)$ (so that $a=n-k+1$ and $b=1$)  and consider parking functions in the $K\times k$ rectangle. The corresponding symmetric function $E_{K,k}\cdot 1$ has degree $K$.

\begin{rmk}
For $k=n$ we have $K=k$, and the symmetric function in question is $E_{n,n}\cdot 1=\nabla e_n$.
\end{rmk}

\subsection{Delta Conjecture}

A second generalization of the Shuffle Theorem called the {\bf Delta Conjecture} was formulated by Haglund, Remmel, and Wilson \cite{HRW}. It involves a more general Macdonald eigenoperator $\Delta'_{e_{k-1}}$ and relates it to \emph{stacked parking functions} (see Figure \ref{fig:map-F} at right for a stacked parking function).  The (Rise) Delta Conjecture, in terms of stacked parking functions, states 
\[
\Delta'_{e_{k-1}}e_n = \sum_{P\in \LD^{\mathrm{stack}}_{n,k}} t^{\area(P)}q^{\hdinv(P)} x^P.
\]
This version of the Delta Conjecture was proven by~\cite{DAdderio} and independently by~\cite{BHMPS}. To the authors' knowledge, the Valley version (involving a statistic $\mathrm{wdinv}$) of the conjecture remains open.

Our next main result is a combinatorial proof of the skewing identity starting from the Rise Delta and Rational Shuffle combinatorial formulas.

\begin{thm}\label{thm: combi main}
We have
$$
s_{(k-1)^{n-k}}^{\perp}\sum_{\pi\in \WPF_{K,k}}t^{\area(\pi)}q^{\dinv(\pi)}x^{\pi}=\sum_{P\in \LD_{n,k}^{\stack}} t^{\area(P)} q^{\hdinv(P)} x^P,
$$  where the sums are over word parking functions.   In other words, Theorem \ref{thm:IntroSkewing}
holds combinatorially.
\end{thm}

\begin{figure}
    \centering
    \includegraphics{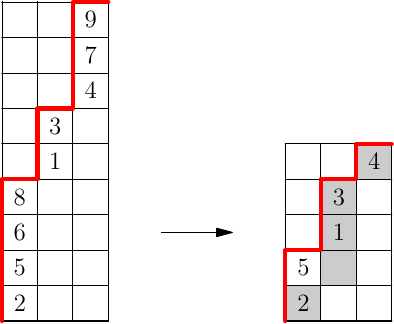}
    \caption{Mapping a standard $(K,k)$-parking function to a stacked parking function by removing the big labels, that is, those that are greater than $n$.  Here $k=3$ and $n=5$, so $K=9$.}
    \label{fig:map-F}
\end{figure}

There are several major steps to proving this directly.  We first convert the stacked parking function formula into modified $(K,k)$-rational parking functions by considering ``big'' and ``small'' letters on the latter and making use of a weight-preserving reduction map. There are $n$ ``small" labels and $K-n=(n-k)(k-1)$ ``big" labels, and all big labels are strictly bigger than the small ones.
If the word parking function has all distinct labels from $1$ to $K$ (so that this is a \textit{standard} parking function), we simply consider ``small'' letters to be less than or equal to $n$ and ``big'' letters strictly greater than $n$. See Figure \ref{fig:map-F}, and see Sections \ref{sec: stacks} and \ref{sec:reduce} for details and the generalization to word parking functions.  

In Section \ref{sec: LLT} we then compare the coefficients of powers of $t$ on both sides, and express them in terms of specific \textit{LLT polynomials} \cite{HHLRU,LLT} of the form 
$$
f_{D,\underline{b}}=\sum_{\pi_{\bg}\in \WPF(D,\underline{b})}q^{\tdinv_{\bg}(\pi_\bg)}x^{\pi_{\bg}},\quad f_{D,\underline{s}}=\sum_{\pi_{\sm}\in \WPF(D,\underline{s})}q^{\tdinv_{\sm}(\pi_\sm)}x^{\pi_{\sm}}
$$
Here $\WPF(D,\underline{b})$ indicates the set of partial word parking functions for a sequence of a designated number of ``big'' boxes $b$ at the top of each column of a rational Dyck path $D$.  Similarly $\WPF(D,\underline{s})$ is the set of partial word parking functions on the ``small'' boxes (the remaining boxes in each column).  

Two short computations prove the identities
$$
s_{(k-1)^{n-k}}^{\perp}\sum_{\pi\in \WPF_{K,k}}t^{\area(\pi)}q^{\dinv(\pi)}x^{\pi}=\sum_{D,\underline{b},\underline{s}}\langle  f_{D,\underline{b}}[X;q], s_{(k-1)^{n-k}}\rangle f_{D,\underline{s}}.
$$
and
$$
\sum_{P\in \LD_{n,k}^{\stack}} t^{\area(P)} q^{\hdinv(P)} x^P=\sum_{D,\underline{b},\underline{s}}q^{c_{D,\underline{b}}}f_{D,\underline{s}}
$$
where $c_{D,\underline{b}}$ is a constant defined in terms of the combinatorial statistics.  This reduces the problem to showing that the following inner product is this single power of $q$: $$\langle  f_{D,\underline{b}}[X;q], s_{(k-1)^{n-k}}\rangle = q^{c_{D,\underline{b}}}.$$
(See Theorem \ref{thm: perping big main}). We show this by expressing $s_{(k-1)^{n-k}}$ in terms of an alternating sum of homogeneous symmetric functions using the \textit{Jacobi-Trudi identity}, which expresses the inner product as a sum of monomials given by the combinatorial LLT expansion.  Finally, we use a sign-reversing involution to cancel all but one term in the resulting sum (Section \ref{sec: sign reversing}), and use an inductive argument to show that the remaining term has the correct power (Section \ref{sec: coefficient}).

Finally, this analysis leads to a new proof of the Rise Delta Theorem as follows.

\begin{cor}\label{cor: commutative-diagram}
    The Rise Delta Theorem (bottom horizontal arrow below) follows from the Rational Shuffle Theorem (top horizontal arrow) and Theorem \ref{thm:IntroSkewing} (left vertical arrow) via a direct combinatorial proof (right vertical arrow, see Section \ref{sec: combinatorics}).
$$
\begin{tikzcd}
E_{K,k}\cdot 1 \arrow{rr}  \arrow{d}{s_{(k-1)^{n-k}}^{\perp}} & &
 \displaystyle\sum_{\pi\in \WPF_{K,k}}t^{\area(\pi)}q^{\tdinv(\pi)}x^{\pi}\arrow{d}{s_{(k-1)^{n-k}}^{\perp}} \arrow{ll} \\
 \Delta'_{e_{k-1}}(e_n) \arrow{rr} & &
\displaystyle\sum_{\pi\in \mathrm{WLD}^{\mathrm{stack}}_{n,k}}t^{\area(\pi)}q^{\hdinv(P)}x^{\pi} \arrow{ll}
\end{tikzcd}
$$

\end{cor}

\subsection{Organization of the paper}

The paper is organized as follows. In Section \ref{sec: background}, we give a combinatorial background on symmetric functions, parking functions and affine permutations. In Section \ref{sec: skewing}, we discuss the operators from the Elliptic Hall Algebra and their properties, and  prove Theorem \ref{thm:IntroSkewing}. 
Finally, we give our direct combinatorial proof of Theorem \ref{thm:IntroSkewing} that links the combinatorics of the Rise Delta Theorem and the Rational Shuffle Theorem in Section \ref{sec: combinatorics}.

\section*{Acknowledgments}

We thank Fran\c{c}ois Bergeron, Eric Carlsson, Mark Haiman, Jim Haglund, Oscar Kivinen, Jake Levinson, Misha Mazin, Anton Mellit, Andrei Negu\cb{t}, Anna Pun, George Seelinger, and Andy Wilson for useful discussions.

\section{Notation and Background}
\label{sec: background}

\subsection{Symmetric functions}
\label{sec: symmetricfunctions}

We refer to \cite{Macdonald} for details on many of the standard definitions in this section.

We will work with rings of symmetric functions both in $k$ and in infinitely many variables, denoted respectively by $\Lambda_k$ and $\Lambda$.   There is a natural projection $\pi_k:\Lambda\to \Lambda_k$ which sends $f\in \Lambda$ to $f(x_1,\ldots,x_k,0,\dots)$.  

We will use elementary and complete homogeneous symmetric functions
$$
e_m=\sum_{i_1<\cdots<i_m}x_{i_1}\cdots x_{i_m},\quad h_m=\sum_{i_1\le \cdots\le i_m}x_{i_1}\cdots x_{i_m}.
$$
Note that $\pi_k(e_m)=0$ for $m>k$. For a partition $\lambda$ we write
$$
e_{\lambda}=\prod_i e_{\lambda_i},\quad  
h_{\lambda}=\prod_i h_{\lambda_i}.
$$
We also have monomial symmetric functions 
$$
m_{\lambda}= \sum_{(i_1,\ldots,i_\ell)} x_{i_1}^{\lambda_1}x_{i_2}^{\lambda_2}\cdots x_{i_\ell}^{\lambda_{\ell}}
$$
where $(i_1,\ldots,i_\ell)$ is any tuple of distinct positive integers such that $i_j< i_{j+1}$ whenever $\lambda_j=\lambda_{j+1}$.

We also use the Schur functions $s_{\lambda}$, which may be defined in terms of the Jacobi-Trudi formula $s_\lambda=\det((h_{\lambda_i+i-j})_{i,j=1}^\ell)$.  They can alternatively be defined in terms of the monomial symmetric functions, via the formula
$s_\lambda=\sum K_{\lambda\mu} m_\mu$ where the coefficients $K_{\lambda\mu}$ are the \textit{Kostka numbers}, which count the number of column-strict \textit{Young tableaux} of shape $\lambda$ and content $\mu$.  We draw our Young tableaux in French notation:
$$\young(4,25,1155)$$
and the above tableau has content $(2,1,0,1,3)$, with the $i$th entry indicating the multiplicity of $i$ in the tableau.   Its shape is $(4,2,1)$, indicating the length of each row from bottom to top.

Note that 
$
e_\ell=m_{(1^{\ell})}=s_{(1^\ell)},\ h_\ell=s_{(\ell)}.
$
We will denote these symmetric functions and their images under $\pi_k$ in the same way, if the number of variables is clear from the context. Note that 
\begin{equation}
\label{eq: projected Schur}
\pi_k(s_{\lambda})=0\ \text{if and only if}\ \ell(\lambda)>k,
\end{equation}
and $\pi_k(s_{\lambda})=s_{\lambda}(x_1,\ldots,x_k)$ for $\ell(\lambda)\le k$ form a basis of $\Lambda_k$. 

\begin{defn}
The Hall inner product on $\Lambda$ is defined by 
$$
\langle s_{\lambda},s_{\mu}\rangle=\delta_{\lambda,\mu}.
$$
Similarly, the Hall inner product on $\Lambda_k$ is defined by 
$$
\langle s_{\lambda},s_{\mu}\rangle_k=\delta_{\lambda,\mu}\ \text{if}\ \ell(\lambda),\ell(\mu)\le k.
$$
\end{defn}

Note that in general
$
\langle f,g\rangle\neq \langle\pi_k(f),\pi_k(g)\rangle_k
$
since for $f=g=s_{\lambda}$ with $\ell(\lambda)>k$ the left hand side equals 1 and the right hand side vanishes. Nevertheless, we have the following:
\begin{prop}
\label{prop: Hall restriction}
Suppose $f\in \Lambda$ and 
$
g\in \Span\{s_{\lambda}:\ell(\lambda)\le k\}\subset \Lambda.
$
Then $\langle f,g\rangle=\langle\pi_k(f),\pi_k(g)\rangle_k$.
\end{prop}

It will be useful to consider an extension of $\Lambda_k$ to the space $\Lambda_k^{\pm}$ of symmetric {\em Laurent} polynomials in $x_1,\ldots,x_k$. Any element of $\Lambda_k^{\pm}$ can be written as $(x_1\cdots x_k)^{-s}f(x_1,\ldots,x_k)$ for some $s\ge 0$ and $f\in \Lambda_k$, so $\Lambda_k^{\pm}$ is a localization of 
$\Lambda_k$ in $(x_1\cdots x_k)$. 
Since $s_{\lambda+(1^{k})}=(x_1\cdots x_k)s_{\lambda}$, we get the following decomposition of $\Lambda_k^{\pm}$:
\begin{equation}
\label{eq: lambda plus minus}
\Lambda_k^{\pm}=\Lambda_k\oplus \Span\left\{(x_1\cdots x_k)^{-s}s_{\lambda}\ :\ s>0,\ell(\lambda)<k\right\}.
\end{equation}
Given $f\in \Lambda_{k}^{\pm}$ we denote by $f_{\pol}$ the projection of $f$ on $\Lambda_k$ along the decomposition \eqref{eq: lambda plus minus}.

Alternatively, one can think of $\Lambda_k^{\pm}$ as the space of characters of finite dimensional rational $\GL_k$-representations, with basis
$$
(x_1\cdots x_k)^{-s}s_{\lambda}(x_1,\dots, x_k)=\ch\left[ (\det)^{-s}\otimes V_{\lambda}\right].
$$
Here $\det$ is the one-dimensional determinant representation of $\GL_k$, $V_{\lambda}$ is the irreducible representation of $\GL_k$ with highest weight $\lambda$, and $(\det)^{-s}=(\det^*)^{\otimes s}$. The elements of $\Lambda_k$ correspond to {\em polynomial representations} of $\GL_k$.
Note that $\mathrm{ch}(V^*)=\mathrm{ch}(V)(x_1^{-1},\ldots,x_k^{-1})$ and 
$\ch(V_1\otimes V_2)=\ch(V_1)\cdot \ch(V_2).$

\begin{prop}
\label{prop: Hall constant term}~

\begin{enumerate}[(a)]
\item For representations $U,V$ of $\GL_k$ we have 
$$
\langle \ch(U),\ch(V)\rangle_{k}=\dim \Hom_{\GL_k}(U,V)=\dim (U^*\otimes V)^{\GL_k}.
$$
\item For symmetric functions $f,g\in \Lambda_k$ we have
$$
\langle f,g\rangle_{k}=\langle x^0\rangle f(x_1^{-1},\ldots,x_k^{-1})g(x_1,\ldots,x_k)
$$
where $\langle x^0\rangle$ denotes the constant term.
\item For $f\in \Lambda_k$ and $g\in \Lambda_k^{\pm}$ we have
$$
\langle f,g_{\pol}\rangle_{k}=\langle x^0\rangle f(x_1^{-1},\ldots,x_k^{-1})g(x_1,\ldots,x_k)
$$
\end{enumerate}
\end{prop}
\begin{proof}
Part (a) follows from Schur's Lemma and part (b) follows from (a) by linearity. To prove (c), we decompose 
$g=g_{\pol}+g_{-}$ as in \eqref{eq: lambda plus minus}, then
$
\langle x^0\rangle f(x_1^{-1},\ldots,x_k^{-1})g_{-}(x_1,\ldots,x_k)=0.
$
Therefore,
\begin{align*}
\langle f,g_{\pol}\rangle_{k}&=\langle x^0\rangle f(x_1^{-1},\ldots,x_k^{-1})g_{\pol}(x_1,\ldots,x_k)\\
&=\langle x^0\rangle f(x_1^{-1},\ldots,x_k^{-1})g(x_1,\ldots,x_k).\qedhere
\end{align*}
\end{proof}
 
Next, we would like to discuss multiplication operators and their adjoints with respect to the Hall inner product. 

\begin{defn}
The operators $s_{\lambda}^{\perp}:\Lambda\to \Lambda$ and $s_{\lambda}^{\perp,k}:\Lambda_k\to \Lambda_k$ are defined such that the identities
$$
\langle s_{\lambda}^{\perp}f,g\rangle=\langle f,s_{\lambda}g\rangle,\quad \langle s_{\lambda}^{\perp,k}f,g\rangle_k=\langle f,s_{\lambda}g\rangle_k
$$
hold for all symmetric functions $f,g\in \Lambda$ (resp. $f,g\in \Lambda_k$).
\end{defn}

\begin{lem} 
\label{lem: perp projection}
Assume $\ell(\lambda)\le k$ and suppose $f\in \Span\{s_{\nu}:\ell(\nu)\le k\}\subset \Lambda$. Then 
$$
s_{\lambda}^{\perp}f=\sum_{\ell(\mu)\le k} c_{\mu} s_{\mu}
$$
where
$$
c_{\mu}=\langle x^0\rangle s_{\mu}(x_1^{-1},\ldots,x_k^{-1})s_{\lambda}(x_1^{-1},\ldots,x_k^{-1})f(x_1,\ldots,x_k).
$$
Furthermore, $$\pi_k(s_{\lambda}^{\perp}f)=s_{\lambda}^{\perp,k}(\pi_k(f)).$$
\end{lem}

\begin{proof}
We have $s_{\lambda}^{\perp}s_{\nu}=\sum_{\mu}c_{\lambda,\mu}^{\nu}s_{\mu}$ where $c_{\lambda,\mu}^{\nu}$ are Littlewood-Richardson coefficients. If $\ell(\nu)\le k$ then the   coefficients $c_{\lambda,\mu}^{\nu}$ could be nonzero only when $\ell(\lambda),\ell(\mu)\le k$. This implies
$$
s_{\lambda}^{\perp}s_{\nu}\in \Span\{s_{\mu}:\ell(\mu)\le k\},
$$
and the same holds for $s_{\lambda}^{\perp}f$ if $f\in \Span\{s_{\nu}:\ell(\nu)\le k\}.$

To determine the coefficients $c_{\mu}$, we write
\begin{align*}
c_{\mu}&=\langle s_{\mu},s_{\lambda}^{\perp}f\rangle=
\langle s_{\lambda}s_{\mu},f\rangle=\langle\pi_k(s_{\lambda}s_{\mu}),\pi_k(f)\rangle_k\\
&=\langle x^0\rangle s_{\mu}(x_1^{-1},\ldots,x_k^{-1})s_{\lambda}(x_1^{-1},\ldots,x_k^{-1})f(x_1,\ldots,x_k).
\end{align*}
The third equation follows from Proposition \ref{prop: Hall restriction} and the last equation follows from Proposition \ref{prop: Hall constant term}.
\end{proof}

\begin{lem}
Suppose that $f\in \Lambda_k$ then 
$$
s_{\lambda}^{\perp,k}f=\left[s_{\lambda}(x_1^{-1},\ldots,x_k^{-1})f(x_1,\ldots,x_k)\right]_{\pol}.
$$
\end{lem}
\begin{proof}
By Proposition \ref{prop: Hall constant term} we have for all $g\in \Lambda_k$
$$
\langle g,\left[s_{\lambda}(x_1^{-1},\ldots,x_k^{-1})f(x_1,\ldots,x_k)\right]_{\pol}\rangle_{k}=$$
$$\langle x^0\rangle g(x_1^{-1},\ldots,x_k^{-1})s_{\lambda}(x_1^{-1},\ldots,x_k^{-1})f(x_1,\ldots,x_k)=
\langle gs_{\lambda},f\rangle_{k},
$$
and the result follows.
\end{proof}

\subsection{Rational parking functions}

We consider rational Dyck paths of height $K$ and width $k$, that stay weakly above the northeast diagonal in the grid.  A \textbf{word parking function} is a labeling of the vertical runs of the Dyck path by positive integers such that the labeling strictly increases up each vertical run (but letters may repeat between columns; hence ``word'' parking function).  We let $\WPF_{K,k}$ be the set of word parking functions whose path is a rational Dyck path in the $K\times k$ grid.

\begin{defn}
 The \textbf{area} of an element of $\WPF_{K,k}$ is the number of whole boxes lying between the path and the diagonal, so that the diagonal does not pass through the interior of the box.
\end{defn}

\begin{defn}
The \textbf{dinv} statistic (for ``diagonal inversions'') on $\WPF_{K,k}$ is defined as $$\dinv(P)=\pathdinv(D)+\tdinv(P)-\maxtdinv(D)$$ where $D$ is the Dyck path of $P$. We define each of these three quantities separately below.
 \end{defn}
 To define $\pathdinv(D)$, recall that the \textit{arm} of a box above a Dyck path in the $K\times k$ grid is the number of boxes to its right that still lie above the Dyck path.  The \textit{leg} is the number of boxes below it that still lie above the Dyck path.  

 Below, we use Cartesian coordinates in the first quadrant for both the boxes and points in the grid, where the coordinates of a box are the coordinates of its lower left corner.  In particular,  the lower left-most point (and box) in the grid has coordinates $(0,0)$.
\begin{defn}\label{def:arm-leg}
    The \textbf{pathdinv} of Dyck path $D$ from $(0,0)$ to $(k,K)$ (with $k|K$) is the number of boxes $b$ above the Dyck path of $P$ for which $$\frac{\arm(b)}{\leg(b)+1}\leq k/K< \frac{\arm(b)+1}{\leg(b)}$$
\end{defn}

\begin{rmk}
    We call this statistic $\pathdinv$ here to emphasize that it only depends on the path shape, and to distinguish it from $\dinv$ of the parking function.  It was simply called $\dinv$ in \cite{RationalShuffle}.
\end{rmk}

    A \textbf{diagonal} in the $K$ by $k$ grid, where $k|K$, is a set of boxes that pairwise differ by an integer multiple of the vector $(1,K/k)$. The \textbf{main diagonal} of the $K$ by $k$ rectangle is the line between $(0,0)$ and $(k,K)$.

\begin{defn}
    A pair of boxes $a,b$ in a $K$ by $k$ grid (with $k|K$) is an \textbf{attacking pair} if and only if either:
    \begin{itemize}
        \item $a$ and $b$ are on the same diagonal, with $a$ to the left of $b$, or 
        \item $a$ is one diagonal below $b$, and to the right of $b$.
    \end{itemize}
\end{defn}

\begin{rmk}
    When boxes are known to be labeled, we often use the name of the box and its label interchangeably, as in the definition below. 
\end{rmk}

\begin{defn}\label{def:tdinv}
The \textbf{tdinv} of $P\in \WPF_{K,k}$ is the number of attacking pairs of labeled boxes $a,b$ in $P$ such that $a<b$.  
\end{defn}

\begin{defn}\label{def:maxtdinv}
    The \textbf{maxtdinv} of a Dyck path $D$ is the largest possible $\tdinv$ of any parking function of shape $D$.  
\end{defn}

\begin{example}
      The right hand diagram in Figure \ref{fig:PF-example} has the following $\tdinv$ pairs: 
      \begin{itemize}
          \item Two pairs of the form $(2,3)$ on the main diagonal
          \item Pairs $(2,5), (3,5), (3,6)$ between the main diagonal and the second diagonal.
          \item A pair $(5,6)$ on the second diagonal.
          \item A pair $(5,6)$ between the second and third diagonal.
          \item A pair $(1,8)$ between the third and fourth diagonal.
      \end{itemize}
    Thus $\tdinv$ is equal to 8 in this rational parking function.  To compute $\maxtdinv$, we include all other attacking pairs of labels, even if they are not in increasing order.  Thus $\maxtdinv$ is $16$.
\end{example}

Note that $\mathrm{maxtdinv}(D)$ can be achieved by taking $\mathrm{tdinv}(D)$ of the labeling of $D$ obtained by labeling across diagonals starting from the bottom-most diagonal and moving upwards. Thus, $\mathrm{maxtdinv}(D)$ is the number of attacking pairs $a,b$ such that $a$ and $b$ are boxes directly to the right of an up-step of $D$.

\section{Skewing formula}
\label{sec: skewing}

\subsection{Delta and Shuffle Theorems}

We denote by $\Lambda(q,t)$ and $\Lambda_k(q,t)$ the spaces of symmetric functions in infinitely many (resp. $k$) variables with coefficients in rational functions in $q$ and $t$. We can extend the Hall inner product from $\Lambda$ to $\Lambda(q,t)$, and from $\Lambda_k$ to $\Lambda_k(q,t)$ and apply the results of Section \ref{sec: background} verbatim.

The following result was conjectured in \cite{HRW} and proved independently in \cite{BHMPS} and \cite{DAdderio}:

\begin{thm}[Rise Delta Theorem~\cite{BHMPS,DAdderio}] \label{thm:rise}
We have
\[
\Delta'_{e_{k-1}} e_n = \sum_{P\in \LD_{n,k}^{\mathrm{stack}}} q^{\area(P)} t^{\hdinv(P)} x^P.
\]
\end{thm}

 See Section \ref{sec: stacks} for definitions of the combinatorial objects and statistics on the right-hand side above.

Here the $\Delta'_{e_{k-1}}$ is the operator on $\Lambda(q,t)$ which is diagonal in the modified Macdonald basis $\widetilde{H}_{\lambda}(x;q,t)$ \cite{HHL} with eigenvalues
$$
\Delta'_{e_{k-1}}\widetilde{H}_{\lambda}=e_{k-1}[B'_{\lambda}]\widetilde{H}_{\lambda},\quad B'_{\lambda}=\sum_{\square\in \lambda,\,\square\neq (0,0)}q^{a'(\square)}t^{\ell'(\square)}
$$
Here $e_{k-1}[B_\lambda']$ denotes the operation of evaluating $e_{k-1}$ at the $|\lambda|-1$ terms of $B_\lambda'$, where we consider $e_{k-1}$ as a symmetric function in $|\lambda|-1$ variables.  We also write $a'(\square)$ and $\ell'(\square)$ to denote the co-arm and co-leg of the box $\square$ respectively, which are the number of boxes to their left and below.  At $k=n$ we have 
$$
e_{k-1}[B'_{\lambda}]=\prod_{\square\in \lambda}q^{a'(\square)}t^{\ell'(\square)},
$$
so that $\Delta'_{e_{n-1}}$ coincides with the celebrated $\nabla$ operator \cite{Nabla} and Delta Theorem specializes to the Shuffle Theorem conjectured in \cite{HHLRU} and proved in \cite{CM}. 

We will need a generalization of the Shuffle Theorem known as the Compositional Rational Shuffle Theorem, conjectured in \cite{RationalShuffle} and proved in \cite{Mellit}. To state it, we need to recall some constructions related to the {\bf Elliptic Hall Algebra} $\cE_{q,t}$. 

The algebra $\cE_{q,t}$ has generators $P_{a,b},(a,b)\in \mathbb{Z}^2$ satisfying certain complicated relations \cite{EHA}. We will not need these relations but record some useful properties:
\begin{itemize}
\item[(a)] If $(a',b')=(ca,cb)\in \mathbb{Z}^2$ for some rational constant $c>0$, then $[P_{a,b},P_{a',b'}]=0$. In particular, for each pair $(a,b)$ with $\mathrm{gcd}(a,b)=1$ (or, equivalently, for each {\em slope} $b/a\in \mathbb{Q}$) there is a commutative subalgebra of $\cE_{q,t}$ generated by $P_{ca,cb}$ for all integers $c\ge 1$.

\item[(b)] A certain extension of the group $\mathrm{SL}(2,\mathbb{Z})$ acts on $\cE_{q,t}$ by algebra automorphisms \cite[Corollary 3.9,Lemma 5.3]{EHA}. If $M\in \mathrm{SL}(2,\mathbb{Z})$ then the corresponding automorphism sends the generator $P_{a,b}$ to $P_{M(a,b)}$, up to a certain monomial in $q,t$.
\item[(c)] There is an anti-automorphism $\psi$ of $\cE_{q,t}$ such that $\psi(P_{a,b})=P_{b,a}$.   
\item[(d)] The algebra $\cE_{q,t}$ acts on $\Lambda(q,t)$. The operator $P_{a,b}$ has degree $a$, that is,
$$
\deg P_{a,b}(f)=\deg f+a.
$$
The operators $P_{a,0}$ act on $\Lambda(q,t)$ by multiplication by power sums $p_a$ (up to a scalar factor).
\end{itemize}

Given a symmetric function $F\in \Lambda(q,t)$, we can transform it to an operator $F_{b/a}$ in $\cE_{q,t}$ as follows: first expand $F$ in power sums $p_i$, then replace each $p_i$ by $P_{ia,ib}\in \cE_{q,t}$. Since $P_{ia,ib}$ pairwise commute, we obtain a well-defined element of $\cE_{q,t}$ of slope $b/a$. Alternatively, we can find $M\in \mathrm{SL}(2,\mathbb{Z})$ 
such that $M(1,0)=(a,b)$, then 
the corresponding automorphism of $\cE_{q,t}$ sends $F$ (thought of as a multiplication operator and hence an element of $\cE_{q,t}$ of slope zero) to $F_{b/a}$.

\begin{defn}
Suppose $\mathrm{GCD}(a,b)=d$. We define the operator $E_{a,b}\in \cE_{q,t}$ as the result of rotation of the elementary symmetric function $e_d$ to slope $b/a$ as above.
\end{defn}

In particular, $E_{a,0}$ is the operation of multiplication by the elementary symmetric function $e_a$.  The formula for the symmetric function $E_{a,b}\cdot 1$ was conjectured   in \cite[Conjecture 3.2]{RationalShuffle}   and proved in \cite{Mellit}, and we restate its specialization to the case $(a,b)=(K,k)$ here.

\begin{thm}[Rational Shuffle Theorem \cite{Mellit}] \label{thm:E}
    $$E_{K,k}\cdot 1 = \sum_{P\in \WPF_{K,k}} q^{\area(P)}t^{\dinv(P)}x^P.$$
\end{thm}

\subsection{Shuffle algebra expressions}

We will also need another incarnation of $\cE_{q,t}$ that is known as the {\em Shuffle algebra}  $\cS$. In short, $\cS=\bigoplus_{m}\cS_m$ is spanned by symmetric Laurent polynomials $f(x_1,\ldots,x_m)$ satisfying so-called ``wheel conditions''. The multiplication is given by the {\em shuffle product}
$$
f(x_1,\ldots,x_m)\star g(x_1,\ldots,x_{\ell})=\sum_{w\in S_{m+\ell}}w\left[f(x_1,\ldots,x_m) g(x_{m+1},\ldots,x_{m+\ell})\prod_{\stackrel{1\le i\le m,}{m+1\le j\le k+\ell}}\Gamma(x_i/x_j)\right]
$$
where 
$$
\Gamma(x)=\frac{(1-qtx)}{(1-x^{-1})(1-qx)(1-tx)}.
$$
We refer to \cite{BHMPS,BHMPS2,Negut} on more details and a specific isomorphism relating $\cS$ and $\cE_{q,t}$. In our normalization of the shuffle product and the isomorphism we follow the conventions of \cite{BHMPS}. In particular, we have the following.

\begin{thm}\cite[Proposition 6.13]{Negut}
\label{thm: negut}
Under the isomorphism relating $\cS$ and $\cE_{q,t}$, the element $E_{a,b}\in \cE_{q,t}$ corresponds to 
$$
\phi_{a,b}=\frac{x_1^{S_1}\cdots x_a^{S_a}}{\prod_{i=1}^{a-1} (1-qt x_i/x_{i+1})}\in \cS_{a}
$$
where 
$$
S_i=\left\lceil \frac{ib}{a}\right\rceil-\left\lceil \frac{(i-1)b}{a}\right\rceil.
$$
\end{thm}

Here we use the conventions of \cite[Proposition 3.6.1]{BHMPS2} which is slightly different from the original conventions of \cite{Negut} due to the different normalization of shuffle product.

Let $f(x_1,\ldots,x_k)$ be a Laurent polynomial (or power series) in $x_1,\ldots,x_k$. We define

$$
\sigma(f)=\sum_{w\in S_k}w\left(\frac{f}{\prod_{i<j}(1-x_j/x_i)}\right),
$$
and 
$$
H_{q,t}^{k}(f)=\sigma\left(\frac{f\prod_{i<j}(1-qtx_i/x_j)}{\prod_{i<j}(1-qx_i/x_j)(1-tx_i/x_j)}\right).
$$
For any $f$ the function $H_{q,t}^{k}(f)$ is symmetric in $x_1,\ldots,x_k$.
Following \cite{BHMPS,BHMPS2} we will always implicitly expand the denominators as geometric series
$$
\frac{1}{(1-qx_i/x_j)}=\sum_{b=0}^{\infty} q^bx_i^{b}x_j^{-b},\ \frac{1}{(1-tx_i/x_j)}=\sum_{b=0}^{\infty} t^bx_i^{b}x_j^{-b}
$$
and interpret $H_{q,t}^{k}(f)$ as a Laurent power series in a certain completion of $\Lambda_k^{\pm}$. As before, we denote by $H_{q,t}^{k}(f)_{\pol}$ the projection to (a certain completion of) $\Lambda_k$. 
We will need the following easy observation:
\begin{prop}
\label{prop: hqt symmetric}
Assume that $h$ is a symmetric function in $x_1,\ldots,x_k$ and $f$ is arbitrary. Then 
$$
H_{q,t}^k(hf)=hH_{q,t}^k(f).
$$
\end{prop}

The following result relates the operator $H_{q,t}^{k}$ to shuffle algebra $\cS$.

\begin{thm}\cite[Proposition 3.4.2]{BHMPS}
\label{thm: psi eval shuffle}
Suppose $f$ is an element of the shuffle algebra $\cS\simeq \cE_{q,t}$ and, as above, $\psi$ is an anti-automorphism of $\cE_{q,t}$ such that $\psi(P_{a,b})=P_{b,a}$. Then 
$$
\pi_k(\omega \psi(f)(1))=(\omega\psi(f))(1)(x_1,\ldots,x_k)=H_{q,t}^k(f)_{\pol}.
$$
\end{thm}

(In the statement above, recall that we write $\pi_k:\Lambda\to \Lambda_k$ for the restriction to the first $k$ variables.)  Next, we write the expressions for the Delta conjecture from \cite{BHMPS}.

\begin{thm}\cite[Theorem 4.4.1]{BHMPS}
For $0\le l<m\le N$ we have
$$
(\omega \Delta_{h_l}\Delta'_{e_{m-l-1}}e_{N-l})(x_1,\ldots,x_m)=H_{q,t}^{m}(\phi(x))_{\pol},
$$
where 
$$
\phi(x)=\frac{x_1\cdots x_m}{\prod (1-qt x_i/x_{i+1})}h_{N-m}(x_1,\ldots,x_m)e_{l}(x_2^{-1},\ldots,x_m^{-1}).
$$
\end{thm}

\begin{cor}
\label{cor: delta shuffle}
Setting $l=0$, $m=k$, and $N=n$, we get
$$
\pi_k(\omega \Delta'_{e_{k-1}}e_{n})=(\omega \Delta'_{e_{k-1}}e_{n})(x_1,\ldots,x_k)=H_{q,t}^{k}\left(\frac{x_1\cdots x_kh_{n-k}(x_1,\ldots,x_k)}{\prod (1-qt x_i/x_{i+1})}\right)_{\pol}.
$$
\end{cor}

\begin{lem}
\label{lem: E shuffle}
We have
$$
\pi_k(\omega E_{K,k}(1))=(\omega E_{K,k}(1))(x_1,\ldots,x_k)=H_{q,t}^{k}\left(\frac{x_1^{n-k+1}\cdots x_k^{n-k+1}}{\prod (1-qt x_i/x_{i+1})}\right)_{\pol}.
$$
\end{lem}

\begin{proof}
We have $\psi(E_{K,k})=E_{k,K}$. By Theorem \ref{thm: negut}
the operator $E_{k,K}\in \cE_{q,t}$ corresponds to the   element
$$
\phi_{k,K}=\frac{x_1^{S_1}\cdots x_k^{S_k}}{\prod (1-qt x_i/x_{i+1})}\in \cS
$$
where 
$$
S_i=\left\lceil \frac{iK}{k}\right\rceil-\left\lceil \frac{(i-1)K}{k}\right\rceil=(n-k+1)=\frac{K}{k}.
$$
Now by  Theorem \ref{thm: psi eval shuffle} we get
$$
\pi_k(\omega E_{K,k}(1))=\pi_k(\omega \psi(\phi_{k,K})(1)))=H_{q,t}^{k}(\phi_{k,K})_{\pol}
$$
and the result follows.
\end{proof}

\subsection{From Shuffle conjecture to Delta conjecture}

We now connect the two formulas with a skewing operator.  First note the following lemma for the rectangular Schur function in $k$ variables.

\begin{lem}
We have
$$
s_{(n-k)^{k-1}}(x_1,\ldots,x_k)=\sum_{\stackrel{\mu_i\le n-k}{|\mu|=(n-k)(k-1)}}x_1^{\mu_1}\cdots x_k^{\mu_k}.
$$
\end{lem}

\begin{proof}
We have
$$
s_{(n-k)^{k-1}}=\sum_{\mu}K_{(n-k)^{k-1},\mu}m_{\mu},
$$
where $K_{(n-k)^{k-1},\mu}$ is the Kostka number computing the number of column-strict tableaux of shape $(n-k)^{k-1}$ and content $\mu$. If $\ell(\mu)>k$, then $m_{\mu}(x_1,\ldots,x_k)$ vanishes and we can ignore all such terms. If $\ell(\mu)\le k$, then by \cite[Lemma 3.5]{GG} we have 
$$
K_{(n-k)^{k-1},\mu}=\begin{cases}
1 & \text{if\ all}\ \mu_i\le n-k\\
0 & \text{otherwise}
\end{cases}
$$
which gives the resulting formula.
\end{proof}

\begin{cor}
\label{cor: dual rectangle}
We have
$$
s_{(n-k)^{k-1}}(x_1^{-1},\ldots,x_k^{-1})=\frac{h_{n-k}(x_1,\ldots,x_k)}{x_1^{n-k}\cdots x_k^{n-k}}.
$$
\end{cor}
\begin{proof}
Given a monomial $x_1^{\mu_1}\cdots x_k^{\mu_k}$ with $0\le \mu_i\le n-k$ and $\sum \mu_i=(n-k)(k-1)$, we write $\alpha_i=n-k-\mu_i$. Note that $0\le \alpha_i\le n-k$ and $\sum \alpha_i=(n-k)k-(n-k)(k-1)=n-k$, so that 
$$
s_{(n-k)^{k-1}}(x_1^{-1},\ldots,x_k^{-1})=\sum_{\mu}x_1^{-\mu_1}\cdots x_k^{-\mu_k}=\frac{\sum_{\alpha}x_1^{\alpha_1}\cdots x_k^{\alpha_k}}{x_1^{n-k}\cdots x_k^{n-k}}=\frac{h_{n-k}(x_1,\ldots,x_k)}{x_1^{n-k}\cdots x_k^{n-k}}
$$ as desired.
\end{proof}

\begin{lem}
\label{lem: perp Hqt}
Let $\varphi(x_1,\ldots,x_k)$ be an arbitrary Laurent polynomial. Then we have the identity:
$$
s_{(n-k)^{k-1}}^{\perp,k}H_{q,t}^k(\varphi(x))_{\pol}=H_{q,t}^{k}\left(\frac{\varphi(x)h_{n-k}(x_1,\ldots,x_k)}{x_1^{n-k}\cdots x_k^{n-k}}\right)_{\pol}.
$$
\end{lem}

\begin{proof}
We pair both sides with $s_{\lambda}$ using the Hall inner product on $\Lambda_k$.  On the left-hand side we have
\begin{align*}
\langle s_{\lambda},s_{(n-k)^{k-1}}^{\perp}H_{q,t}^k(\varphi(x))_{\pol}\rangle_k &=\langle s_{\lambda}s_{(n-k)^{k-1}},H_{q,t}^k(\varphi(x))_{\pol}\rangle_k \\
&= \langle x^0\rangle s_{\lambda}(x_1^{-1},\ldots,x_k^{-1})s_{(n-k)^{k-1}}(x_1^{-1},\ldots,x_k^{-1})H_{q,t}^k(\varphi(x)).
\end{align*}
The last equation follows from Proposition \ref{prop: Hall constant term}(c).
In the right-hand side by Propositions \ref{prop: Hall constant term}(c) and \ref{prop: hqt symmetric} we have  
$$
\left\langle s_{\lambda},H_{q,t}^{k}\left(\frac{\varphi(x)h_{n-k}(x_1,\ldots,x_k)}{x_1^{n-k}\cdots x_k^{n-k}}\right)_{\pol}\right\rangle_k=
\langle x^0\rangle s_{\lambda}(x_1^{-1},\ldots,x_k^{-1})H_{q,t}^{k}\left(\frac{\varphi(x)h_{n-k}(x_1,\ldots,x_k)}{x_1^{n-k}\cdots x_k^{n-k}}\right)=
$$
$$
\langle x^0\rangle s_{\lambda}(x_1^{-1},\ldots,x_k^{-1})\frac{h_{n-k}(x_1,\ldots,x_k)}{x_1^{n-k}\cdots x_k^{n-k}}H_{q,t}^{k}(\varphi(x)).
$$
Now the statement follows from Corollary \ref{cor: dual rectangle}.
\end{proof}

We conclude our main skewing formula as follows.

\begin{mainthm}
We have
$$
s_{(k-1)^{n-k}}^{\perp}E_{K,k}
(1)=\Delta'_{k-1}(e_n).
$$
\end{mainthm}

\begin{proof}
First, let us prove that the restrictions of both sides to $k$ variables agree.

Consider the function $\varphi(x_1,\ldots,x_k)=\frac{x_1^{n-k+1}\cdots x_k^{n-k+1}}{\prod (1-qt x_i/x_{i+1})}$. By Lemmas  \ref{lem: E shuffle} and\ref{lem: perp Hqt}   we get
$$
s_{(n-k)^{k-1}}^{\perp,k}\pi_k(\omega E_{K,k}(1))=
s_{(n-k)^{k-1}}^{\perp,k}H_{q,t}^{k}(\varphi)_{\pol}=H_{q,t}^{k}\left(\frac{\varphi(x)h_{n-k}(x_1,\ldots,x_k)}{x_1^{n-k}\cdots x_k^{n-k}}\right)_{\pol}=
$$
$$
H_{q,t}^{k}\left(\frac{x_1\cdots x_k h_{n-k}(x_1,\ldots,x_k)}{\prod (1-qt x_i/x_{i+1)})}\right)_{\pol}
$$
which agrees with $\pi_k(\omega \Delta'_{e_{k-1}}e_n)$ by Corollary \ref{cor: delta shuffle}.

Next, we need to argue that all Schur functions appearing in the both sides of the equations have at most $k$ parts, so that we do not lose any information when restricting to $k$ variables. For the right hand side, it follows from \cite[Remark 4.4.2]{BHMPS}. 

For the left hand side, by e.g. \cite[Corollary 3.7.2]{BHMPS2}  (or by Rational Shuffle Theorem) we have that all Schur functions appearing in $\omega E_{K,k}(1)$ have at most $k$ parts. Now by Lemma \ref{lem: perp projection} we get
$$
\pi_k(s_{(n-k)^{k-1}}^{\perp}\omega E_{K,k}(1))=s_{(n-k)^{k-1}}^{\perp,k}\pi_k(\omega E_{K,k}(1))=
\pi_k(\omega \Delta'_{e_{k-1}}e_n)
$$
and by the above
$$
s_{(n-k)^{k-1}}^{\perp}\omega E_{K,k}(1)=\omega \Delta'_{e_{k-1}}e_n.
$$
Applying $\omega$ to both sides (which transposes the partition for the skewing operator) we obtain the result.
 \end{proof}

\begin{example}
    For $(n,k) = (3,2)$, then $K=4$ and we have
    \[
    E_{4,2}(1) = s_{(2,2)} + (q+t)s_{(2,1,1)} + (q^2+qt+t^2)s_{(1,1,1,1)}
    \]
    and
    \[
    \Delta'_{e_{1}}e_3 =  (1+q+t)s_{(2, 1)}+ (q+t+q^2+qt+t^2)s_{(1, 1, 1)}.
    \]
    We apply $s_{(1)}^\perp$ to the former to obtain the latter.
\end{example}

\section{Combinatorial proof of the skewing formula}\label{sec: combinatorics}
In this section, we give a combinatorial proof of Theorem~\ref{thm: perping big main}.  Throughout this section, we say a parking function of any type is \textbf{standard} if its labels are $1,2,3,\ldots,m$ for some $m$, each occuring exactly once.  A \textbf{word parking function} is a generalization of a standard parking function in which labels may occur with higher multiplicity, but columns still must be strictly increasing.

\subsection{Stacks}
\label{sec: stacks}

We use the notation of \cite{HRW} here, and recall the definition of a \textit{stacked} parking function that can be used to reformulate the Rise version of the Delta 
Conjecture.  

\begin{defn}
A \textbf{stack} $S$ of boxes in an $n\times k$ grid is a subset of the grid boxes such that there is one element of $S$ in each row, at least one in each column, and each box in $S$ is weakly to the right of the one below it.

A (word) \textbf{stacked parking function} with respect to $S$ is a labeled up-right path $D$ such that each box of $S$ lies below $D$, and the labeling is strictly increasing up each column.
\end{defn}

We write $\LD(S)$ is the set of stacked parking functions with respect to $S$, and
\[
\LD_{n,k}^{\stack} \coloneqq \bigcup_{S\in \Stack_{n,k}} \LD(S).
\]

\begin{figure}
    \centering
    \includegraphics{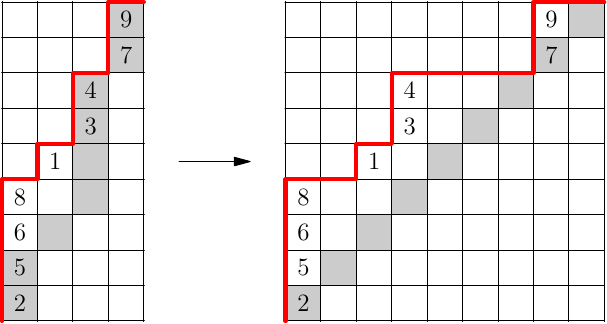}
    \caption{At left, a stacked parking function $P$ for $k=4$, $n=9$.  At right, the corresponding ordinary parking function $P'$ of size $n$.}
    \label{fig:stack}
\end{figure}

\begin{defn}
     The \textbf{area} of an element of $\LD(S)$ is the number of boxes between the path and the stack $S$.   
\end{defn}

In Figure \ref{fig:stack}, the area of the parking function is $4$.

\begin{defn}
\label{def: dinv for stacks}
    The \textbf{hdinv} statistic on $\LD$ is defined as follows.  Given $P\in \LD^{\stack}_{n,k}$, consider the stack heights $w_1,w_2,\ldots,w_k$ of each column of $P$.  Insert $w_i-1$ empty columns between column $i$ and $i+1$ for each $i$ from right to left, and connect the new gaps in the Dyck path with horizontal lines, as in Figure \ref{fig:stack}.  This forms an ordinary parking function $P'$ on a square grid, and $$\hdinv(P)\coloneqq\dinv(P'),$$ where $\dinv(P')$ is the number of pairs of labeled boxes $(a,b)$ with labels $(\alpha,\beta)$ with $\alpha\le \beta$ and either:
\begin{itemize}
    \item $a,b$ in the same diagonal with $b$ to the right,
    \item $a$ in one diagonal lower and to the right.
\end{itemize}
\end{defn}
\begin{rmk}
    The map $P\to P'$ described above is the map $\phi^{-1}_{n,n-k}$ map from \cite[page 9]{HRW}.
\end{rmk}

For example, the stacked parking function $P$ in Figure \ref{fig:stack} has $\hdinv(P)=5$, because the ordinary parking function at right has five diagonal inversions: $(2,7),(5,9),(1,3),(1,8),(3,8)$.

\subsection{Reducing to stacks}\label{sec:reduce}
Given a \textit{standard} parking function in the $K\times k$ rectangle, that is, using the labels $1,2,\ldots,K$ each exactly once, we call a label $a$ {\bf big} if $a>n$, and {\bf small} if $a\le n$. Then there are $n$ small labels and $K-n$ big labels in total, and we write $b_i$ for the number of big labels in column $i$.  We say it is \textbf{admissible} if $b_i\le n-k$ for all $i$, and we note there is a natural map $F$ from admissible standard $(K,k)$ parking functions to the set $\LD_{n,k}^\stack$ as follows:

\begin{itemize}
\item The parking function $F(\pi)$ is obtained by erasing all big labels in $\pi$ and deleting all vertical steps of the Dyck path to the left of these labels.
\item The heights of the stacks are given by $w_i=n-k+1-b_i$.
\end{itemize}

See Figure \ref{fig:map-F} from the introduction for an example.

We need to generalize this map to \textit{word} parking functions and stacks.  To do so, we define two new sets of combinatorial objects.

\begin{figure}
    \centering
    \includegraphics{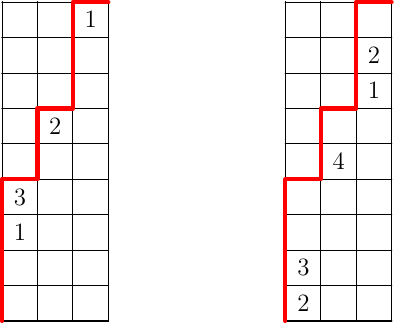}
    \caption{At left, an element of $\WPF(D,\underline{b})$ for $K=9,k=3$ where $\underline{b}=(2,1,1)$.  At right, an element of $\WPF(D,\underline{s})$ where $\underline{s}=(2,1,2)$ is the corresponding composition of small labels.  Note that $\underline{b}$ is admissible because $b_i\leq n-k=2$ for all $i$, and $\sum b_i=4=9-5=K-n$.}
    \label{fig:big-and-small}
\end{figure}

\begin{defn}
    Let $D$ be a $(K,k)$ Dyck path, and let $\underline{b}=(b_1,\ldots,b_k)$ be a sequence of numbers such that $b_i$ is less than or equal to the number of vertical steps just left of column $i$ in $D$ for each $i$, and $\sum b_i=K-n$.  Then we define $\WPF(D,\underline{b})$ to be the set of all column-strict labelings of the top $b_i$ boxes in the $i$th column under $D$ for each $i$. 

    Similarly, if $\underline{s}=(s_1,\ldots,s_k)$ is a sequence of numbers with the same column by column restriction and $\sum s_i=n$, then we define $\WPF(D,\underline{s})$ to be the set of all column-strict labelings of the bottom $s_i$ boxes to the right of each vertical run in $D$. 

    We say $\underline{b}$ is \textbf{admissible} if $b_i\le n-k$ for all $i$, and we say $\underline{s}$ is admissible if the corresponding big sequence - formed by the complements of $s_i$ relative to the heights of the vertical runs of $D$ - is admissible.  (See Figure \ref{fig:big-and-small})
\end{defn}

We generalize the map $F$ to this setting as follows.

\begin{defn}
    We define $$F:\bigcup_{D}\bigcup_{\underline{s}\text{ admissible}} \WPF(D,\underline{s})\to \LD_{n,k}^\stack$$ by 
    \begin{itemize}
        \item The parking function $F(\pi)$ is obtained by deleting all vertical steps of the Dyck path above the small labeled letters.
        \item The heights of the stacks are given by $w_i=n-k+1-b_i$, where $\underline{b}$ is the big sequence corresponding to $\underline{s}$.
    \end{itemize}
\end{defn}

\begin{lem}
\label{lem: admissible to stacks}
The map $F$ is a well-defined bijection.
\end{lem}

\begin{proof}
Since $\pi$ is admissible, we have $b_i\le n-k$ and $w_i>0$. We check that the stacks add up to $n$. Using the above definition, which says that $w_i=n-k+1-b_i$,
$$
w_1+\cdots+w_k=k(n-k+1)-\sum b_i=K-\sum b_i=n.$$
Next, we need to check that the Dyck path is above the stack. This is equivalent to
\begin{equation}\label{eq:aboveinequality}
s_1+\cdots+s_i\ge w_1+\cdots+w_i=i(n-k+1)-(b_1+\cdots+b_i).
\end{equation}
Indeed, $(b_j+s_j)_j$ are the lengths of the vertical runs of the original Dyck path, so we have
$$
(b_1+s_1)+\cdots+(b_i+s_i)\ge i(n-k+1),
$$ 
which is equivalent to \eqref{eq:aboveinequality}. Hence $F$ is well-defined.

To show it is bijective, we show we can reverse it.  Consider a stacked parking function $S$ in $\LD_{n,k}^\stack$.  We can recover the unique parking function $\pi$ such that $F(\pi)=S$ as follows.  The sequence $\underline{s}$ is simply given by the heights of the vertical runs of the Dyck path of $S$.  The sequence $\underline{b}$ is given by the heights of the stack and the formula $b_i=n-k+1-w_i$.  The numbers $s_i+b_i$ determine the heights of the vertical runs of the Dyck path of $\pi$, recovering $D$, and this lies above the diagonal by \eqref{eq:aboveinequality}.  Finally, the labeling of $\pi$ is precisely the labeling of $S$, placed on the bottommost $s_i$ letters of each vertical run in $\pi$.  This completes the proof.
\end{proof}

\begin{lem}
\label{lem: area match}
We have $\area(\pi)=\area(F(\pi))$, where the two area statistics are the appropriate ones for each object.
\end{lem}

\begin{proof}
Indeed, the number of boxes in the $i$-th column of $\pi$ which contribute to its area equals 
$$
(b_1+s_1)+\cdots+(b_i+s_i)-i(n-k+1),
$$
while the number of area boxes in the $i$-th column of $F(\pi)$ equals
$$
(s_1+\cdots+s_i)-(w_1+\cdots+w_i)=(s_1+\cdots+s_i)-i(n-k+1)+(b_1+\cdots+b_i)
$$
and the result follows.
\end{proof}

Following Definition \ref{def: dinv for stacks}, we can relate the labeled boxes in $P\in \LD_{n,k}^{\stack}$ and $P'\in \WPF_{n,n}$
as follows.  Given a box $A=(i,x)$ containing a parking function label in $P$, we define
$$
A'=(i+(w_1-1)+\cdots+(w_{i-1}-1),x)=(w_1+\cdots+w_{i-1}+1,x).
$$
where $h_i$ is the number of stacked boxes in the $i$-th column.
Then we move the label from box $A$ into box $A'$ in the $n\times n$ square.

\begin{lem}
\label{lem: tdinv match}
Suppose $A$ and $B$ are two labeled boxes in the $(K,k)$ small parking function $\pi\in \WPF_{D,\underline{s}}$, let $F(A)$ and $F(B)$ be their images under $F$ in $F(\pi)$ and 
$F(A)'$ and $F(B)'$ the corresponding boxes in $F(\pi)'$. 
Then $A$ and $B$ form an attacking pair if and only if $F(A)'$ and $F(B)'$ do.
\end{lem}

\begin{proof}
Suppose that $A=(i,x)$ and $B=(j,y)$ and $i<j$. They are on the same diagonal if and only if
\begin{equation}
\label{eq: same diagonal}
y-x=(n-k+1)(j-i).
\end{equation}
We have
$$
F(A)=(i,x-(b_1+\cdots+b_{i-1})),\ F(B)=(j,y-(b_1+\cdots+b_{j-1}))
$$
and
$$
F(A)'=(w_1+\cdots+w_{i-1}+1,x-(b_1+\cdots+b_{i-1})),\ F(B)'=(w_1+\cdots+w_{j-1}+1,y-(b_1+\cdots+b_{j-1})).
$$
Now
$$
(y-(b_1+\cdots+b_{j-1}))-(x-(b_1+\cdots+b_{i-1}))=y-x-(b_i+\cdots+b_{j-1})
$$
while 
$$
(w_1+\cdots+w_{j-1}+1)-(w_1+\cdots+w_{j-1}+1)=w_i+\cdots+w_{j-1}=(j-i)(n-k+1)-(b_i+\cdots+b_{j-1}).
$$
By \eqref{eq: same diagonal}, $A$ and $B$ are on the same diagonal if and only if $F(A)$ and $F(B)$ are on the same diagonal. The same proof shows that if $A$ and $B$
 are on neighboring diagonals then $F(A)'$ and $F(B)'$ are on the neigboring diagonals as well.
\end{proof}

\begin{cor}\label{cor: tdinv hdinv}
    The statistic $\hdinv(F(\pi))$ is equal to $\tdinv_\sm(\pi)$, the number of $\tdinv$'s between the small labels in $\pi$. 
\end{cor}

\subsection{Proof setup and outline} \label{sec: LLT}

Returning to the motivation from standard parking functions, given a $(K,k)$-parking function $\pi$ with numbers $1,2,\ldots,K$ used exactly once, we write:
\begin{itemize}
\item $\tdinv_{\sm}(\pi)$ is the number of diagonal inversions between the boxes with small labels.
\item $d(\underline{s},\underline{b})$ is the number of pairs $(c,c')$ such that $c$ is a small box, $c'$ is a big box, and $(c,c')$ are attacking (in that order).
\item $\tdinv_{\bg}(\pi)$ is the number of diagonal inversions between the boxes with big labels.
\end{itemize}
Note that $d(\underline{s},\underline{b})$ depends only on the positions of big and small labels, but not on a specific parking function $\pi$. Also,
\begin{equation}
\label{eq: tdinv big and small}
\tdinv(\pi)=\tdinv_{\sm}(\pi)+d(\underline{s},\underline{b})+\tdinv_{\bg}(\pi).
\end{equation}

We generalize these notions to word parking functions as follows. Given a $K\times k$ Dyck path $D$ and any weak composition $\underline{b} = (b_1,\dots, b_k)$ of $(n-k)(k-1)$ such that $b_i$ is at most the number of vertical steps in column $i$ of $D$, we will say the \textbf{big boxes} of $D$ are the top $b_i$ boxes in column $i$ that are immediately to the right of a vertical step of $D$. Similarly, the \emph{small boxes} of $D$ are the remaining boxes to the right of vertical steps. We introduce symmetric functions 
$$
f_{D,\underline{b}}=\sum_{\pi_{\bg}\in \WPF(D,b)}q^{\tdinv_{\bg}(\pi_\bg)}x^{\pi_{\bg}},\quad f_{D,\underline{s}}=\sum_{\pi_{\sm}\in \WPF(D,s)}q^{\tdinv_{\sm}(\pi_\sm)}x^{\pi_{\sm}}
$$
where $\tdinv_\bg$ measures the number of diagonal inversions on big labels of $\pi_\bg$, and $\tdinv_\sm$ measures the number of inversions on small labels of $\pi_\sm$.  We also still write $d(\underline{s},\underline{b})$ for the number of attacking pairs between big and small boxes such if the big box is labeled with a larger number than the small box, it would form an inversion. Note that all of these statistics and polynomials are defined for general $\underline{b}$, not necessarily admissible.

Finally, we define the statistic
\begin{equation} \label{eq:constant}
    c_{D,\underline{b}} = \maxtdinv(D) - \pathdinv(D) - d(\underline{s},\underline{b}).
\end{equation}

The following is our main combinatorial result. The proof relies on several combinatorial constructions which we outline and prove in detail in Sections \ref{sec: sign reversing} and \ref{sec: coefficient}, but we provide pinpoint references to these results here.

\begin{thm}
\label{thm: perping big main}
Let $D$ be a Dyck path and $\underline{b}$ be an admissible sequence (so $b_i\leq n-k$ and at most the number of vertical steps of $D$ in column $i$). Then
    \[
    \langle \omega f_{D,\underline{b}}[X;q], s_{(n-k)^{k-1}}\rangle =\langle  f_{D,\underline{b}}[X;q], s_{(k-1)^{n-k}}\rangle = q^{c_{D,\underline{b}}}.
    \]
     If $\underline{b}$ is not admissible, then $\langle f_{D,\underline{b}}[X;q], s_{(k-1)^{n-k}}\rangle = 0$.
\end{thm}

\begin{proof}
The proof goes in several steps.  

{\bf Step 1:} Given a composition $\talpha=(\talpha_1,\ldots,\talpha_\ell)$, the pairing $\langle h_{\talpha}, f_{D,\underline{b}}[X;q]\rangle$ equals the coefficient of $f_{D,\underline{b}}[X;q]$ at the monomial symmetric function $m_{\talpha}$, which counts the column-strict fillings of $(D,\underline{b})$ with content $\talpha$.
We denote the set of such labelings by $P_{D,\underline{b},\talpha}$ and write
\begin{equation}
\label{eq: monomial coefficient}
\langle h_{\talpha}, f_{D,\underline{b}}[X;q]\rangle=\sum_{P\in P_{D,\underline{b},\talpha}}q^{\tdinv_{\bg}(P)}
\end{equation}
Note that $(D,\underline{b})$ has at most $k$ vertical runs, so in a column-strict filling of $(D,\underline{b})$ any label is repeated at most $k$ times. Therefore, for $\talpha_i\ge k+1$  there are no such fillings.

{\bf Step 2:} We expand the Schur function using the Jacobi-Trudi formula, where we replace any $h_j$ with $0$ if $j\ge k+1$:
\begin{equation}\label{eq:det truncated}
    s_{(k-1)^{n-k}}=\det \begin{pmatrix}
    h_{k-1} & h_k & 0 & 0 & \cdots & 0 \\
    h_{k-2} & h_{k-1} & h_k & 0 & \cdots & 0 \\
    h_{k-3} & h_{k-2} & h_{k-1} & h_k & \cdots & 0 \\
    \vdots &          &         & \ddots & \ddots & 0 \\
    h_{2k-n+1} & \ldots &       &        & h_{k-1} & h_k \\
    h_{2k-n} & \ldots &         &        &        & h_{k-1}
\end{pmatrix}
\mod (h_{j},j\ge k+1).
\end{equation}

For a composition $\talpha$ such that $h_{\talpha} = h_{\talpha_1}\cdot \cdots \cdot h_{\talpha_{n-k}}$ appearing in the expansion of \eqref{eq:det truncated}, we write $\alpha$ to be the complementary composition where $\alpha_i=k-\talpha_i$ for all $i$.
We call all resulting compositions $\alpha$ {\bf allowable contents} (see Definition \ref{def: admissible} and Lemma \ref{lem: allowable-contents}), and rewrite \eqref{eq:det truncated} as
$$
s_{(k-1)^{n-k}}=\sum_{\alpha\ \mathrm{allowable}}(-1)^{\sgn(\talpha)}h_{\talpha} \mod (h_{j},j\ge k+1).
$$
For the definition of $\sgn(\talpha)$, see \eqref{eq: sign alpha}.
By combining this with \eqref{eq: monomial coefficient} and observing that compositions with $\talpha_j\ge k+1$  do not contribute to the sum, we get
\begin{equation}
\label{eq: Jacobi Trudi expanded}
\langle s_{(k-1)^{n-k}}, f_{D,\underline{b}}[X;q]\rangle=\sum_{\alpha\ \mathrm{allowable}}\sum_{P\in P(D,\underline{b},\talpha)}(-1)^{\sgn(\talpha)}q^{\tdinv_{\bg}(P)}.
\end{equation}
Also, since $\talpha$ has $(n-k)$ parts,  the set $P(D,\underline{b},\talpha)$ is empty for all $\talpha$ if $b_i> n-k$ for some $i$ (that is, $\underline{b}$ is not admissible). This implies that $\langle  f_{D,\underline{b}}[X;q], s_{(k-1)^{n-k}}\rangle = 0$ whenever $\underline{b}$ is not admissible which proves the second part of the theorem. From now on we assume $b_i\le n-k.$

{\bf Step 3:} This is the crucial step. In Definition \ref{def: sign reversing} we define a sign-reversing involution $\varphi$  and prove in Theorem \ref{thm: sign reversing} the following:
\begin{itemize}
\item $\varphi$ is an involution  on the set of  column-strict fillings with allowable contents
\item $\varphi$ has a unique fixed point, which we denote $P^0_{D,\underline{b}}$, that has positive sign,
\item $\varphi$ preserves the statistics $\tdinv_{\bg}$   and reverses the sign $(-1)^{\sgn(\talpha)}$ for every element except $P^0_{D,\underline{b}}$.
\end{itemize}

{\bf Step 4:} By the previous step, the terms in \eqref{eq: Jacobi Trudi expanded} cancel in pairs according to the involution $\varphi$, and we are left with a single term corresponding to the fixed point: 
$$
\langle s_{(k-1)^{n-k}},f_{D,\underline{b}}[X;q]\rangle=q^{\tdinv_{\bg}(P_{D,\underline{b}}^0)}.
$$
To complete the proof, we prove in  Theorem \ref{thm: shrunken fixed point} (combined with Lemma \ref{lem: pathdinv}) that $\tdinv_{\bg}(P^0_{D,\underline{b}})=c_{D,\underline{b}}$.  By Equation \eqref{eq:constant}, this translates to showing the combinatorial fact that \begin{equation}\label{eq: fixed-point equality}\tdinv_\bg(P^0_{D,\underline{b}})=\maxtdinv(D)-\pathdinv(D)-d(\underline{s},\underline{b})
\end{equation}
This is highly nontrivial and occupies most of Subsection \ref{sec: coefficient} (see Theorem \ref{thm: shrunken fixed point}).
\end{proof}

Assuming all of the above steps, the main combinatorial result can now be proven.

\begin{combimain}
We have
$$
s_{(k-1)^{n-k}}^{\perp}\sum_{\pi\in \WPF_{K,k}}t^{\area(\pi)}q^{\dinv(\pi)}x^{\pi}=\sum_{P\in \mathrm{WLD}_{n,k}^{\stack}} t^{\area(P)} q^{\hdinv(P)} x^P.
$$  where the sums are over column-strict parking functions that may have repeats between columns.  In other words,
$$s_{(k-1)^{n-k}}^\perp E_{K,k}\cdot 1 = \Delta'_{k-1} e_n$$
holds combinatorially.
\end{combimain}

\begin{proof}
We fix a $K\times k$ Dyck path $D$ and denote 
$$
f_{D}=\sum_{\pi\in \WPF_{K,k}(D)}q^{\tdinv(\pi)}x^{\pi}.
$$
where the sum is over column-strict word parking functions.

It is easy to see from \eqref{eq: tdinv big and small} (compare \cite[Equation (11)]{BHMPS3}) that
$$
f_D[X+Y;q]=\sum_{\underline{b},\underline{s}}q^{d(\underline{s},\underline{b})}f_{D,\underline{s}}[X;q]f_{D,\underline{b}}[Y;q]
$$
where the sum is over all possible decompositions of vertical steps of $D$ into big and small.  
Now
$$
s_{(k-1)^{n-k}}^{\perp}f_D=\langle s_{(k-1)^{n-k}}(Y),f_D[X+Y;q]\rangle = \sum_{\underline{b},\underline{s}}q^{d(\underline{s},\underline{b})}f_{D,\underline{s}}[X;q]\langle s_{(k-1)^{n-k}}(Y),f_{D,\underline{b}}[Y;q]\rangle.
$$
By Theorem \ref{thm: perping big main} we can rewrite this as
\begin{equation}
\label{eq: perping fD}
s_{(k-1)^{n-k}}^{\perp}f_D=\sum_{\underline{b},\underline{s}}q^{d(\underline{s},\underline{b})+c_{D,\underline{b}}}f_{D,\underline{s}}[X;q]=\sum_{\underline{b},\underline{s}}q^{\maxtdinv(D)-\pathdinv(D)}f_{D,\underline{s}}[X;q].
\end{equation}
where the sum is over admissible decompositions $(\underline{b},\underline{s})$. Therefore
\begin{align*}
s_{(k-1)^{n-k}}^{\perp}\sum_{\pi\in \WPF_{K,k}}t^{\area(\pi)}q^{\dinv(\pi)}x^{\pi} 
&= 
s_{(k-1)^{n-k}}^{\perp}\sum_{D\in \Dyck(K,k)}t^{\area(D)}q^{\pathdinv(D)-\maxtdinv(D)}f_D \\
&=
\sum_{D\in \Dyck(K,k)}t^{\area(D)}q^{\pathdinv(D)-\maxtdinv(D)}s_{(k-1)^{n-k}}^{\perp}f_D \\
&=
\sum_{D\in \Dyck(K,k)}\sum_{\underline{b},\underline{s}}t^{\area(D)}f_{D,\underline{s}}.
\end{align*}
Here the first equation follows from the definitions
$$
\area(\pi)=\area(D),\ \dinv(\pi)=\pathdinv(D)+\tdinv(\pi)-\maxtdinv(D),
$$
the second equation is the linearity of $s_{(n-k)^{k-1}}^{\perp}$ and the last equation follows from \eqref{eq: perping fD}.

By Lemma \ref{lem: admissible to stacks} there is a bijection $F$ between small parking functions $\pi$ in some $\WPF(D,\underline{s})$ for some $D$ and some admissible $\underline{s}$ sequence, and stacked parking functions.  Furthermore, by Lemma \ref{lem: area match} and Corollary \ref{cor: tdinv hdinv} we have 
$$
\tdinv_{\sm}(\pi)=\hdinv(F(\pi)),\hspace{1cm} \area(D)=\area(F(\pi)),
$$
so
\begin{align*}
\sum_{D\in \Dyck(K,k)}\sum_{\underline{b},\underline{s}}t^{\area(D)}f_{D,\underline{s}}
&=\sum_{D\in \Dyck(K,k)}\sum_{\underline{b},\underline{s}}\sum_{\pi\in \WPF(D,\underline{s}))}t^{\area(\pi)}q^{\tdinv_{\sm}(\pi)}x^{\pi} \\
&=
\sum_{F(\pi)\in \LD^{\stack}_{n,k}}t^{\area(F(\pi))}q^{\hdinv(F(\pi))}x^{F(\pi)}. \qedhere
\end{align*}
\end{proof}

\subsection{Sign-reversing involution}
\label{sec: sign reversing}

We now provide the details for Step 3 of the proof of Theorem \ref{thm: perping big main}.  An admissible tuple $\underline{b}$ for the number of big entries in each column, $b_1,\ldots,b_k$, satisfies $$\sum b_i=K-n=k(n-k+1)-n=(k-1)(n-k)$$ and $b_i\le n-k$ for all $i=1,\ldots,k$.  This means that we can form the set of big boxes by first setting the top $n-k$ elements of each column to be big, and then removing a total of exactly $n-k$ of these boxes from the bottoms of the columns. 

Let $\talpha=(\talpha_1,\ldots,\talpha_{n-k})$ be a composition such that $0\le \talpha_i\le k$ and $\sum_{i=1}^{n-k}\talpha_i=(k-1)(n-k)$.
Let $P_{\underline{b},\talpha}$ be a way of filling the $\underline{b}$-labeled boxes in a given Dyck path with content $\talpha$ (that is, the number of $1$'s equals $\talpha_1$, the number of $2$'s equals $\talpha_2$ and so on), with the columns increasing.  We call such a filling a \textbf{$\underline{b}$-parking function}.  Since $\sum \talpha_i=\sum b_i$, the labels in a $\underline{b}$-parking function range from 1 to $n-k$, as opposed to general word parking functions.

In order to simplify the diagrams and notation, we shrink the height of the diagram as follows.

\begin{defn}
    The \textbf{shrunken diagram} corresponding to a $\underline{b}$-parking function $P_{\underline{b},\talpha}$ is obtained by shifting the $i$-th column of big entries down exactly $(n-k+1)(i-1)=(K/k) (i-1)$ steps for each $i$.  We write $\mathrm{shrink}(P_{\underline{b},\talpha})$ for the resulting diagram.  (See Figure \ref{fig:shrunken}).
\end{defn}

\begin{figure}
\begin{center}
    \includegraphics{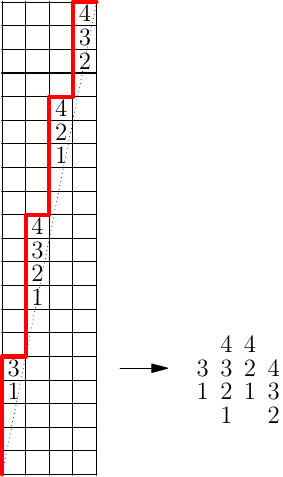}
\end{center}
\caption{\label{fig:shrunken} A $\underline{b}$-parking function and its shrunken diagram.}
\end{figure}

Notice that attacking pairs from the original $\underline{b}$-parking function become either pairs of entries $(x,y)$ in the same row with $x$ left of $y$, or in adjacent rows with $x$ below and to the right of $y$.  

\begin{rmk}These are the same type of attacking pairs as used in the definition of the $\inv$ statistic for the combinatorial formula for Macdonald polynomials in \cite{HHL} (but reversed from left to right).  However, the Macdonald $\inv$ statistic also subtracts the arms of certain boxes.  The set of attacking pairs in which $x<y$ is denoted by $\Inv$ in \cite{HHL}, so we use $\Inv$ here for the \textbf{number} of inversions of a $b$-parking function.  
\end{rmk}

\begin{defn} \label{def:inv}
    We define $\Inv(\mathrm{shrink}(P_{\underline{b},\talpha}))$ to be the number of pairs $(x,y)$ in which $x<y$ and either $x$ is left of $y$ in the same row, or $x$ is right of $y$ and one row below $y$.
\end{defn}

The following observation is now clear.

\begin{prop}
\label{prop: shrink}
    We have $\Inv(\mathrm{shrink}(P_{\underline{b},\talpha}))=\tdinv_\bg(P_{\underline{b},\talpha})$.
\end{prop}

We therefore will replace $\tdinv_\bg$ with $\Inv$ in the shrunken diagram from here on.

\begin{example}
    Consider the parking function in Figure \ref{fig:shrunken} for $n-k=4$ (in this example, $k=4$, $K=20$).   We lower each subsequent column by another $n-k+1=5$ so that the same-diagonal entries are actually on the same row in the diagram at right.  
\end{example}

\begin{defn}
For a column $C$ of big letters of size $b$, define its \textbf{complement} column $\overline{C}$  to be the filling of the $n-k-b$ boxes below $C$ respectively with the complementary set in $\{1,2,\ldots,n-k\}$, increasing from top to bottom. 
\end{defn}

\begin{defn}
If $\talpha=(\talpha_1,\ldots,\talpha_{n-k})$ is a content vector with $0\le \talpha_i\le k$, we define the complementary content $\alpha=(\alpha_1,\ldots,\alpha_{n-k})$ where
\begin{equation}
\label{eq: alpha}
\alpha_i=k-\talpha_i.
\end{equation}
\end{defn}

\begin{lem}
Let $P$ be a filling of $\underline{b}$ with content $\talpha$.  Then its complement $\overline{P}$  has content $\alpha$. 
\end{lem}
\begin{proof}
The union of $P$ and $\overline{P}$ has exactly $k$ columns each containing $1,2,\ldots,n-k$ exactly once, so has content $(k,k,k,\ldots,k)$ of length $n-k$.  Thus the content of $\overline{P}$ plus the content of $P$ is this tuple.
\end{proof}

An example of the complementary fillings of columns are shown in red in Figures \ref{fig:complement-pairs} and \ref{fig:sign-reversing-PF}.  We count inversions slightly differently in the complement $\overline{P}$.

\begin{defn}
 We define $\overline{\Inv}(\overline{P})$ for a complement $\overline{P}$ of a shrunken $\underline{b}$-parking function $P$ to be the number of attacking pairs $(x,y)$ where $x\le y$ (rather than just $x<y$ as in ordinary inversions).  That is, we allow ties in the complement, and we call these inversions \textbf{tied inversions}.

 We similarly write $\Inv(C_1,C_2)$ and $\overline{\Inv}(\overline{C}_1,\overline{C}_2)$ for the number of inversions between two fixed columns, and the number of tied inversions between their complements, respectively.
\end{defn}

\begin{lem}\label{lem:complement-sums}
  Let $C_1,C_2$ be two columns of big letters with $C_1$ left of $C_2$, with fixed heights $h_1$ and $h_2$, and let $\overline{C}_1,\overline{C}_2$ be their complements.   Then $\Inv(C_1,C_2)+\overline{\Inv}(\overline{C}_1,\overline{C}_2)$ is independent of the fillings $C_1$ and $C_2$ and only depends on the heights $h_1$ and $h_2$ and the relative vertical positioning of the columns.  (See Figure \ref{fig:complement-pairs}.)
\end{lem}

\begin{figure}
    \centering
          \includegraphics{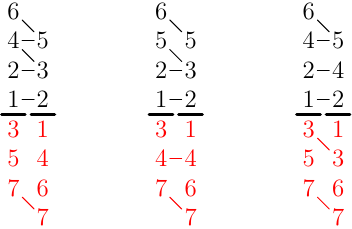}
          \caption{The three pairs of columns shown each have a total of six inversions, when counting $\Inv(C_1,C_2)+\overline{\Inv}(\overline{C}_1,\overline{C}_2)$.}
    \label{fig:complement-pairs}
\end{figure}

\begin{proof}
    Let $b_1,b_2$ be the heights of $C_1,C_2$.  Then each of $C_1,C_2$ is the increasing ordering of some $b_1$-element subset (respectively $b_2$-element subset) of $\{1,2,\ldots,n-k\}$.

    Note that any $b$-element subset can be obtained by starting with $\{1,2,\ldots,b\}$ and performing moves that replace some $i$ in the subset with $i+1$ where $i+1$ is not in the subset.  This means that it suffices to show that such a replacement applied to $C_1$ or $C_2$, forming $C_1',C_2'$ and their complements $\overline{C}_1',\overline{C}_2'$, does not change the total, i.e.\ $\Inv(C_1,C_2)+\overline{\Inv}(\overline{C}_1,\overline{C}_2)=\Inv(C_1',C_2')+\overline{\Inv}(\overline{C}_1',\overline{C}_2')$.

    Let $i$ be in $C_1$ such that $i+1$ is not in $C_1$.  Then when we replace $i$ with $i+1$ in $C_1$, essentially we are switching $i$ in $C_1$ with $i+1$ in $\overline{C}_1$ to form $C_1',\overline{C}_1'$ (and $C_2,\overline{C}_2$ stay the same); all other letters remain in their original positions.  The only inversions that can change are between $i$ in $C_1$ and the entry $x$ in $C_2$ in the same row (if it exists) and $y$ in $C_2$ in the next row below (if it exists).  Similarly, let $z$ and $w$ be the entries in $\overline{C}_2$ in the same row and next row below $i+1$ in $\overline{C}_1$, if they exist.  
    \begin{center}
    \begin{tikzpicture}
    \draw (0,2) node {$i$};
    \draw (1,2) node {$x$};
    \draw (1,1.5) node {$y$};
    \draw (0,1) node {$\vdots$};
    \draw (1,1) node {$\vdots$};
    \draw (0,0) node {\textcolor{red}{$i+1$}};
    \draw (1,0) node {\textcolor{red}{$z$}};
    \draw (1,-0.5) node {\textcolor{red}{$w$}};
    \end{tikzpicture}
    \end{center}
    Notice that there are precisely $i-1$ boxes between the $i$ in $C_1$ and the $i+1$ in $\overline{C}_1$, since by construction those boxes are labeled with $1,2,\ldots,i-1$ in some order.  Thus there are $i-1$ boxes between $x$ and $z$, and between $y$ and $w$.  

    \textbf{Case 1.}  Suppose $x\neq i+1$, $y\neq i$, $z\neq i$, $w\neq i+1$ (this includes the possibility that one or more of these letters does not exist).  Then changing the $i$ in $C_1$ to $i+1$ does not change $\Inv(C_1,C_2)$, and changing the $i+1$ in $\overline{C}_1$ to $i$ does not change $\overline{\Inv}(\overline{C}_1,\overline{C}_2)$.

    \textbf{Case 2.} If $x=i+1$ and $y=i$ both hold, then $z$ and $w$ are not equal to $i$ or $i+1$ (or do not exist) since $\overline{C}_2$ is the complement of $C_2$.  Notice that changing the $i$ to $i+1$ in $C_1$ removes the inversion between the $i$ and $x$, and adds one between the new $i+1$ and $y$, so the total number of diagonal inversions between $C_1,C_2$ remains unchanged, and the total $\overline{\Inv}$ between $\overline{C}_1,\overline{C}_2$ is also unchanged since $z,w$ are not $i$ or $i+1$.  

    \textbf{Case 3.} If $x=i+1$ and $y\neq i$, then $i\not\in C_2$ so $i\in \overline{C}_2$.  Since $x,z$ are $i-1$ spaces apart it follows that $z=i$, and $w\neq i+1$.  Then we see that changing the $i$ in $C_1$ to $i+1$ removes one inversion between $C_1,C_2$ (with $x$) and the change of the $i+1$ in $\overline{C}_1$ to $i$ adds one tied inversion (with $z$) between $\overline{C}_1,\overline{C}_2$.  Thus the total remains unchanged.

    \textbf{Case 4.} If $x\neq i+1$ and $y=i$, now $i+1\not\in C_2$ and by similar reasoning to above we have $w=i+1$ and $z\neq i$, and changing the $i$ to $i+1$ in $C_1$ adds one inversion between $C_1,C_2$ and removes one between $\overline{C}_1,\overline{C}_2$.

    \textbf{Case 5.} If $z=i$ and $w=i+1$, then $x,y$ are not equal to $i,i+1$ (or do not exist), and a similar check to Case 2 shows both Inv totals are unchanged.

    \textbf{Case 6.}  If either $z=i$ and $w\neq i+1$ or $z\neq i$ and $w=i+1$, we are in Case 3 or 4 above by similar reasoning.

    Finally, we need to consider an $i$ in $C_2$ changing to an $i+1$, and compare it to the elements $u,v$ in $C_1$ in the row above and the same row respectively, and compare the $i+1$ in $\overline{C}_2$ changing to an $i$ to the elements $r,s$ in $\overline{C}_1$ that can form an inversion with it.  An exactly analogous casework argument completes the proof in this case as well.
\end{proof}

In light of Lemma~\ref{lem:complement-sums}, in order to find a sign-reversing involution that cancels pairs with same $\Inv$ statistic, we may instead find a sign-reversing involution that cancels pairs with the same $\overline{\Inv}$ statistic on the complement.

We now analyze the possible contents of the complementary diagrams.

\begin{defn}
\label{def: admissible}
   We say a tuple $\alpha=(\alpha_1,\ldots,\alpha_{n-k})$ of nonnegative integers between $0$ and $n-k$ inclusive is an \textbf{allowable content} if it satisfies the following condition: $\alpha_1>0$, and for any positive entry $t=\alpha_i>0$ in $\alpha$ there follows a run of exactly $t-1$ zeroes before the next positive entry.
\end{defn}

 An alternative way of thinking about an allowable content is that for any $i$, the entry $\alpha_i$ preceding any maximal continuous run $\alpha_{i+1}=\cdots = \alpha_{i+t}=0$ of $t$ zeros is equal to $t+1$.

 For an allowable content $\alpha$, we write
 \begin{equation}
 \label{eq: sign alpha}
 \sgn(\alpha)=\sum_{\alpha_i>0}(\alpha_i-1)=\#
\{i:\alpha_i=0\}.
\end{equation}

\begin{example}
    The allowable contents $\alpha$ for $n-k=5$ are:
    \begin{align*}
      (1,1,1,1,1),\qquad (1,1,1,2,0),\qquad(1,1,2,0,1), \qquad(1,2,0,1,1), \\
    (2,0,1,1,1),\qquad(1,2,0,2,0),\qquad (2,0,1,2,0),\qquad(2,0,2,0,1), \\    (1,1,3,0,0),\qquad(1,3,0,0,1),\qquad(3,0,0,1,1),\qquad(2,0,3,0,0),\\
(3,0,0,2,0),\qquad(1,4,0,0,0),\qquad(4,0,0,0,1),\qquad(5,0,0,0,0).
    \end{align*}  Notice that there are $2^{n-k-1}$ allowable contents in general, formed by choosing where the $0$'s are among the entries after $\alpha_1$.
\end{example}

\begin{lem}\label{lem: allowable-contents}
A term $h_{\talpha}$ appears in the determinant 
 $$\det \begin{pmatrix}
 h_{k-1} & h_k & 0 & \ldots & 0\\
h_{k-2} & h_{k-1} & h_k & \ldots & 0\\
\vdots &  & \ddots & & \\
 h_{2k-n} & h_{2k-n+1} & & & h_{k-1}
    \end{pmatrix}$$
    with nonzero coefficient if and only if the complementary term $h_\alpha$ appears in  the determinant
\begin{equation}
\label{eq: dual det}
\det\begin{pmatrix}
      h_{1} & h_{0} & 0 & \cdots & 0\\
h_{2} & h_{1} & h_{0} & \cdots & 0\\
% h_3 & h_2 & h_1 & \cdots & 0 \\
\vdots &  & \ddots & & \\
 h_{n-k} & h_{n-k-1} & & & h_{1}
\end{pmatrix}=\sum_{\alpha\text{ allowable}} (-1)^{\sgn(\alpha)}h_\alpha
\end{equation}
with nonzero coefficient. Here we order the subscripts $\alpha_i$ as they appear in the columns of the matrix from left to right, and we treat the entries $1$ as $h_0$.  For instance, $\alpha$ can be $(2,0,1,1,1,\ldots,1)$. 

The tuples $\alpha$ that appear in this way are precisely the allowable contents for $n-k$, and the sign of $h_\alpha$ is $(-1)^{\sgn(\alpha)}$.
\end{lem}

\begin{proof}
The first statement is clear from the fact that each entry $h_i$ in the first determinant corresponds to entry $h_{k-i}$ in the second determinant (in particular, $h_{k}$ corresponds to $h_0=1$), so we analyze the determinant \eqref{eq: dual det}.  We expand this determinant using the Leibniz formula that says that for an $(n-k)\times (n-k)$ matrix $M$, $$\det(M)=\sum_{\sigma\in S_{n-k}} \sgn(\sigma) \prod_i M_{i,\sigma(i)}.$$

For our matrix, the only nonvanishing terms in the formula above will correspond to permutations whose entries avoid the $0$'s in \eqref{eq: dual det}. For a given such $\sigma$, consider the smallest value $\alpha_1$ for $M_{\alpha_1,\sigma(\alpha_1)}\neq 1$.  Then in particular $\sigma(1)=2$, $\sigma(2)=3$, and so on up to $\sigma(\alpha_1-1)=\alpha_1$, and since $\sigma$ is a permutation with $\sigma(\alpha_1)\neq \alpha_1+1$, we must have $\sigma(\alpha_1)=1$.  Thus $\sigma$ starts with a cycle $(1\,2\,\cdots \,\alpha_1)$ in cycle notation, and has selected the entries in the top left corner of the form:
$$\begin{pmatrix}
    & h_0 &  & &  \\
    &   & h_0 &  &  \\
     &   &  & \ddots &  \\
    &  & & & h_0 \\
    h_{\alpha_1} & & & &
\end{pmatrix}$$
resulting in a factor $(-1)^{\alpha_1-1}h_{\alpha_1}h_0^{\alpha_1-1}$.  Continuing inductively, we see that $\sigma$ is a product of cycles of consecutive elements, each contributing a factor of the form $(-1)^{t-1}h_{t}h_0^{t-1}$ for some $t$. Thus, the resulting term of subscripts on the product of $h_i$'s precisely corresponds to the definition of an allowable sequence.
\end{proof}

\begin{defn}
    The \textbf{reading word} of the complementary filling $\overline{P}$ of a filling $P$ of the big boxes is formed by reading the entries of the top row from right to left, then the second-to-top row from right to left, and so on.  (See Figure \ref{fig:sign-reversing-PF}.)
\end{defn}

\begin{defn}
A \textbf{tied inversion} in a word $w$ is a pair of numbers $i,j$ in $w$ with $i\ge j$ and $i$ to the left of $j$ in $w$.
\end{defn}
The following is clear by the definition of the reading word and diagonal inversions.

\begin{lem}\label{lem:Complement invs}
      If $u,v$ contribute to $\overline{\Inv}$ in a complementary tableau $\overline{P}$, then they form a tied inversion in the reading word as well.  Thus, any operation on $\overline{P}$ that changes an $i+1$ to an $i$, or vice versa, and preserves all tied inversion pairs (and non-inversion pairs) in its reading word also preserves all $\overline{\Inv}$ pairs in $\overline{P}$ itself. 
      
      In particular, such an operation retains the property that the columns of $\overline{P}$ are increasing top to bottom, since these pairs do not form an inversion in the reading word.
\end{lem}

It follows that to complete Step 3 of the proof of Theorem~\ref{thm: perping big main}, it suffices to construct a sign-reversing involution that lowers or raises one letter on the set of all possible reading words with allowable contents, since such an involution will induce an involution on the possible complementary parking functions $\overline{P}$ for any given shape of $\overline{P}$. Our sign-reversing involution will leave one fixed point, the word $123\cdots (n-k)$, which is always the reading word of a unique way of filling any diagram, namely in reading order.  It will also preserve tied inversions, hence canceling all desired terms of the sum. 
This involution then automatically leads to an involution on the $\underline{b}$-parking functions $P$ on a fixed Dyck path $D$ that preserves the number of tied inversions of the complement $\overline{P}$ by Lemma \ref{lem:complement-sums}.

\begin{defn}
\label{def: sign reversing}
    We define a sign-reversing involution $\varphi$ on the set of words $w$ of length $n-k$ with allowable contents, as follows.  

    Let $i$ be the largest entry in $w$ such that (a) there is only one $i$, and (b) if $j$ is the largest letter in $w$ that is less than $i$, then $i$ is to the left of every $j$. 
    Let $m$ be the largest repeated letter in $w$.  Note that $i\neq m$.

    \textbf{Case 1.}  If $i>m$ or $m$ does not exist, replace $i$ with $j$.

    \textbf{Case 2.} If $i<m$ or $i$ does not exist, let $t$ be the smallest letter larger than $m$ in $w$ (or $t=n-k+1$ if $m$ is the largest letter in $w$).  Then replace the first $m$ with $t-1$.

    \textbf{Case 3.} If neither $m$ nor $i$ exist, do nothing.

    The resulting word is $\varphi(w)$.
\end{defn}

\begin{example}
Consider the case $n-k=4$. The determinant we are considering is $$\det \begin{pmatrix}
    h_1 & h_0 & 0 & 0 \\
    h_2 & h_1 & h_0 & 0 \\
    h_3 & h_2 & h_1 & h_0 \\
    h_4 & h_3 & h_2 & h_1
\end{pmatrix}
=h_{1111}-h_{2011}-h_{1201}-h_{1120}+h_{3001}+h_{1300}+h_{2020}-h_{4000}$$
and the pairings of the corresponding words under $\varphi$ are shown in Figure \ref{fig:pairings}.

\begin{figure}
\begin{center}
\begin{tabular}{c c c @{\hspace{0.5cm}}|@{\hspace{0.5cm}} c c c @{\hspace{0.5cm}}|@{\hspace{0.5cm}} c c c}
1243 & $\leftrightarrow$ & 1233 & 3124 & $\leftrightarrow$ & 2124 & 4221 & $\leftrightarrow$ & 2221 \\
1324 & $\leftrightarrow$ &1224 & 3142 &$\leftrightarrow$ & 2142 & 4212 &$\leftrightarrow$ & 2212 \\
1342 & $\leftrightarrow$ &1242 & 3214 &$\leftrightarrow$ & 2214 & 4122 &$\leftrightarrow$ & 2122\\
1423 & $\leftrightarrow$ &1323 & 3241 &$\leftrightarrow$ & 2241 & 1422 &$\leftrightarrow$ & 1222 \\
1432 & $\leftrightarrow$ &1332 & 3412 &$\leftrightarrow$ & 2412 & & & \\
2134 & $\leftrightarrow$ &1134 & 3421 &$\leftrightarrow$ & 2421 & 1143 &$\leftrightarrow$ & 1133 \\
2143 & $\leftrightarrow$ &2133 & 4123 &$\leftrightarrow$ & 3123 & 1413 &$\leftrightarrow$ & 1313 \\
2314 & $\leftrightarrow$ &1314 & 4132 &$\leftrightarrow$ & 3132 & 4113 &$\leftrightarrow$ & 3113 \\
2341 &$\leftrightarrow$ & 1341 & 4213 &$\leftrightarrow$ & 3213 & 1431 &$\leftrightarrow$ & 1331 \\
2413 &$\leftrightarrow$ & 2313 & 4231 &$\leftrightarrow$ & 3231 & 4131 &$\leftrightarrow$ & 3131 \\
2431 &$\leftrightarrow$ & 2331 & 4312 &$\leftrightarrow$ & 3312 & 4311 &$\leftrightarrow$ & 3311 \\ 
     &&   & 4321 &$\leftrightarrow$ & 3321 & & & \\
 4111 &$\leftrightarrow$ & 1111  &&&& 2411 &$\leftrightarrow$ & 1411 \\
    &&   &&&& 2141 &$\leftrightarrow$ & 1141 \\
    && &&&& 2114 &$\leftrightarrow$ & 1114
\end{tabular}
 \end{center}
 \caption{\label{fig:pairings} The pairings of words of length $4$ under the sign-reversing involution.  Not shown: the unique fixed point $1234$ maps to itself.}
 \end{figure}
 \end{example}

 \begin{thm}
 \label{thm: sign reversing}
   The map $\varphi$ (on complement reading words) is a sign-reversing involution on words with allowable content, has exactly one fixed point, and preserves tied inversions. The unique fixed point is the word $12\ldots(n-k)$, which has positive sign.
 \end{thm}

 \begin{proof}
     We first show that $\varphi$ sends a word of allowable content to another word of allowable content.  If it changes $i$ to $j$ as in Case 1, then the content tuple $\alpha$ had $\alpha_i=1$, and then by the definition of allowable content, we must have $(\alpha_j,\ldots,\alpha_i,\alpha_{i+1})=(i-j,0,0,\ldots,0,1,a)$ where $a>0$.  Changing the $i$ to $j$ results in this substring of the content changing to $(i-j+1,0,0,\ldots,0,0,a)$, and so the content is still allowable.  

     Otherwise, if $\varphi$ changes the leftmost $m$ in $w$, and $t$ is the smallest letter larger than $m$, then we have $\alpha_m=t-m$ and a content subsequence $(\alpha_m,\alpha_{m+1},\ldots,\alpha_t)=(t-m,0,0,\ldots,0,0,c)$ where $c$ is the multiplicity of $t$.  When we change an $m$ to $t-1$, the content subsequence changes to $(t-m-1,0,0,\ldots,0,1,c)$, and the content is still allowable.  

 We now show $\varphi$ is an involution. Let $v= \varphi(w)$, and let $i,j,m$ be the corresponding values for $w$ in Definition~\ref{def: sign reversing}, and similarly let $i',j',m'$ be the values for $v$. 

 Before we proceed, observe that $m < t-1$ in Case 2  (so in particular Case 2 does not result in a fixed point) since $m<t-1$ even when $t=n-k+1$, by the definition of allowable content.  
 
 First, suppose that Case 1 applies to $w$, so that $i>m$ and $v$ is formed by changing $i$ to $j$ in $w$.  Then $j$ is a repeated letter in $v$, with the first $j$ being where the $i$ was, and $j>i>m$, so $j$ is now the largest repeated letter in $v$. Thus, $m'=j$.  Since we changed $i$, there is no singleton larger than $i$ satisfying the condition for $i'$ in $v$, so either $i'<j=m'$ or $i'$ does not exist.  Either way, Case 2 applies to $v$.  Since $w$'s content was allowable and $i$ was a singleton, there is an $i+1$ in $w$ (and hence in $v$), and so the smallest letter $t'$ larger than $m'=j$ in $v$ is now $t'=i+1$ (since there are no $i$'s remaining in $v$). Thus, the first appearance of $m'=j$ in $v$ changes back to $i=t-1$ when we apply $\varphi$, and we have $\varphi(\varphi(w))=\varphi(v)=w$.

 Second, suppose Case 2 applies to $w$ so that $v$ is formed by changing the first $m$ to $t-1$, where $t$ is the smallest entry in $w$ larger than $m$ (or otherwise $t=n-k+1$).  We claim that $i'=t-1$, $j'=m$, and $i'>m'$ (or $m'$ does not exist).    In particular, there is only one $t-1$ in $v$ since $t-1$ did not exist in $w$ by the definition of allowable content.  Furthermore, $t-1>m>i$ since we are in Case 2, and so $t-1$ is a larger singleton than $i$ in $v$.  The largest entry smaller than $t-1$ is $m$ by definition, and since we changed the leftmost $m$, we have the $t-1$ is a singleton in $v$ that is to the left of all the copies of the entry just smaller than it (namely, $m$).  Thus $i'=t-1$ and $j'=m$.  Since $i'=t-1>m$, then $i'>m'$ as well (or $m'$ does not exist), so Case 1 applies and $\varphi$ changes $i'=t-1$ back to $m$.  Thus $\varphi$ is an involution.

 We now show $\varphi$ changes sign.   
 Indeed, by \eqref{eq: sign alpha} the sign $\sgn(\alpha)$ equals the number of zeros in $\alpha$.
 Since $\varphi$ changes the number of $0$'s in the content by $\pm 1$, it is a sign-reversing involution.
The unique fixed point is the unique word satisfying Case 3, which must be $w=12\ldots (n-k)$. This fixed point $w$ has content $\alpha=(1,\ldots,1)$ with no zeros, and hence has positive sign.

 Finally, we show $\varphi$ preserves tied inversions.  Since we are always changing a letter to an adjacent-sized letter in order in the word, only the inversions between those two letters can change.  If we are in Case 1 and an $i$ changes to a $j$, then since the $i$ was left of every $j$, it formed an inversion with every $j$ and the new $j$ still forms those inversions since we count tied inversions.   In Case 2, an $m$ changes to a $t-1$. Since we are changing the leftmost $m$, it already forms a tied inversion with the $m$'s to its right, and then so does $t-1$ since $t-1>m$.  
 \end{proof}

\begin{figure}
    \begin{center}
        \includegraphics{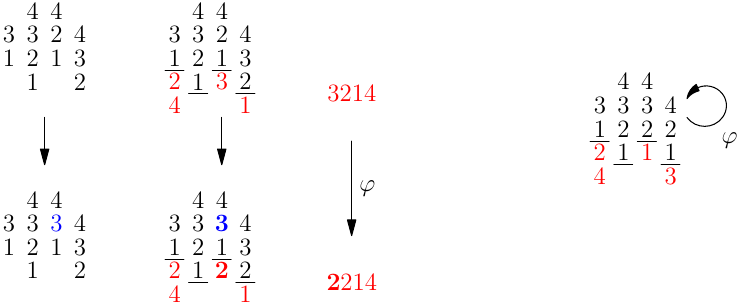}
    \end{center}
    \caption{\label{fig:sign-reversing-PF}At top left, a $\underline{b}$-parking function $P$ written in shrunken notation with $n-k=4$.  From left to right at top we find the complement $\overline{P}$ in red, and the reading word of $\overline{P}$, and then apply $\varphi$  and reinterpret on the parking function to obtain $\varphi(P)$.  At right, the unique fixed point $P$ for this shape.}
\end{figure}

 \begin{cor}
     The map $\varphi$ extends to a sign-reversing involution on the set 
     \[
     \bigsqcup_{\alpha\text{ allowable}} P(D,\underline{b},\widetilde{\alpha})
     \]
     that preserves the statistic $\mathrm{tdinv}_{\mathrm{big}}$ and has a unique fixed point $P_{D,\underline{b}}^0$. Consequently, the right-hand side of \eqref{eq: Jacobi Trudi expanded} is a single term $q^{\tdinv_{\mathrm{big}}(P_{D,\underline{b}}^0)}$.
 \end{cor}

 \begin{proof}
     As Figure~\ref{fig:sign-reversing-PF} illustrates, the induced involution on the set of $\underline{b}$-parking functions on Dyck path $D$ is defined as follows: Given $P$ such a parking function, first find the complement $\overline{P}$ and the reading word $w$ of $\overline{P}$. Then define $\overline{P'}$ to be the unique complement parking function with reading word $\varphi(w)$ on Dyck path $D$. Then let $P$ map to $P'$, the complement of $\overline{P'}$. The desired properties of this map follow by Lemma~\ref{lem:Complement invs} and Theorem~\ref{thm: sign reversing}.

     Since there is a unique fixed point for the complement parking function, namely the filling that has reading word $12\cdots(n-k)$, there is a unique fixed point parking function that we call $P_{D,\underline{b}}^0$.
 \end{proof}

\subsection{Final coefficient}
\label{sec: coefficient}

Let us recall Equation \eqref{eq: Jacobi Trudi expanded}, 
\begin{equation*}
\langle s_{(k-1)^{n-k}}, f_{D,\underline{b}}[X;q]\rangle=\sum_{\alpha\ \mathrm{allowable}}\sum_{P\in P(D,\underline{b},\talpha)}(-1)^{\sgn(\talpha)}q^{\tdinv_{\bg}(P)}%=q^{\tdinv_{\mathrm{big}}(P_{D,\underline{b}}^0)}.
\end{equation*}

Since the sign-reversing involution $\varphi$ cancels all terms in the right-hand sum except the one corresponding to the unique fixed point $P^0_{D,\underline{b}}$, the sum simplifies to  $q^{\tdinv_{\mathrm{big}}(P_{D,\underline{b}}^0)}$.
To complete the proof of Theorem~\ref{thm: perping big main}, we now only need to show that the exponent $\tdinv_{\mathrm{big}}(P_{D,\underline{b}}^0)$ is correct, which we showed in Equation \eqref{eq: fixed-point equality} amounts to proving that
$$\tdinv_\bg(P^0_{D,\underline{b}})=\maxtdinv(D)-\pathdinv(D)-d(\underline{s},\underline{b}).$$

We make a few simplifications to this formula before proving it.  First, if we fill all of the small boxes (which are unfilled in $P^0_{D,\underline{b}}$) with the letter $0$ to form a full parking function $P'$, then $\tdinv_\bg(P^0_{D,\underline{b}})+d(\underline{s},\underline{b})=\tdinv(P')$.  Thus the equation we wish to prove becomes
\begin{equation}
\label{eq:tdinvs}    \tdinv(P')+\pathdinv(D)=\maxtdinv(D),
\end{equation}
which we rephrase and prove in Theorem \ref{thm: shrunken fixed point}.  We first reinterpret each of these statistics on the corresponding shrunken diagrams. 

\begin{defn}
    A \textbf{complete shrunken diagram} (see Figure \ref{fig:genshrunken}) with parameters $n,k$ is a partial labeling of a grid with $k$ columns, and row heights indexed $1,2,3,\ldots$ from bottom to top, such that:
    \begin{itemize}
        \item Each column $C_i$ contains labeled and unlabeled boxes, and the labeled boxes form a finite interval from height $h_i$ to $h_i'$ inclusive with each $h_i'\ge n-k+1$.   The labels have values in $\{0,1,\ldots,n-k\}$, weakly increase from bottom to top, and only the label $0$ may repeat in a column.
        \item We have $h_1=1$ and $h_k'=n-k+1$.  %That is, it starts on the bottom row and the last column ends at height $n-k+1$.
        \item For any $i$, we have $h_{i+1}=h_i'-(n-k)$.  
    \end{itemize}
\end{defn}

\begin{defn}
    The \textbf{complement} of a column $C_i$ in a generalized shrunken diagram is the set $\overline{C}_i$ of letters in $\{1,2,\ldots,n-k\}$ that do not appear in $C_i$, and we draw the letters of each $\overline{C}_i$ in red underneath the nonzero letters of $C_i$, increasing downwards as in Figure \ref{fig:sign-reversing-PF}.  The complement of a generalized shrunken diagram is the union of the columns $\overline{C}_i$.
\end{defn}

\begin{defn}
     We say that a complete shrunken diagram is a \textbf{generalized shrunken fixed point} if the reading word of its complement is strictly increasing.  In particular, all the labels in the complement are distinct.
\end{defn}

It is not hard to see that if the complement reading word is $123\cdots(n-k)$, then we have the result of shrinking one of the fixed points of the sign-reversing involution.  We also write $\mathrm{shrink}(P)$ to denote the shrunken drawing of a parking function $P$, as in the case of $\underline{b}$-parking functions.

\begin{rmk}
In the notations of Section \ref{sec: stacks}, we have
\begin{equation}
\label{eq: hi shrink}
h_i(\mathrm{shrink}(P))=1+(s_1+b_1)+\ldots+(s_{i-1}+b_{i-1})-i(n-k+1)
\end{equation}
and $$h'_i(\mathrm{shrink}(P))=1+(s_1+b_1)+\ldots+(s_{i}+b_{i})-i(n-k+1).$$
A column $C_i$ has $s_i$ zeros and $b_i$ non-zero entries.
\end{rmk}

\begin{figure}
    \centering
 \includegraphics{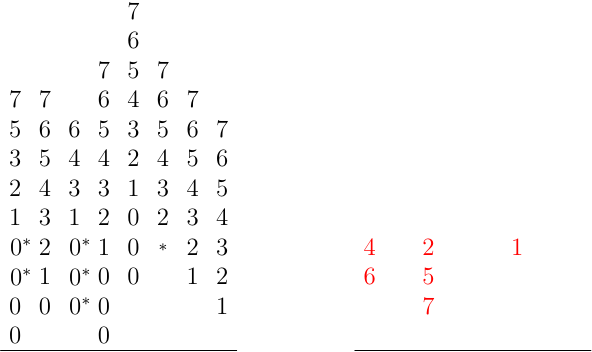} 
    \caption{A generalized shrunken fixed point for $k=8$ and $n=15$ (so $n-k=7$).
    The complement is shown at right, and its reading word is $124567$, which is strictly increasing.  Their corresponding positions are shown at left with $\ast$ symbols.}    \label{fig:genshrunken}
\end{figure}

We now redefine the statistics $\maxtdinv$, $\tdinv$, and $\pathdinv$ for the generalized shrunken fixed points.  Note that \textbf{maxtdinv} simply becomes the number of attacking pairs among labeled boxes, defined as pairs $x,y$ such that either $x$ is left of $y$ in the same row, or $y$ is in the row above $x$ and to the left of $x$.  Then, \textbf{tdinv} is the number of such pairs for which $x<y$.  

\begin{defn}
    We define \textbf{maxInv} of a shrunken diagram to be the maximum number of attacking pairs among labeled entries in the sense of Definition \ref{def:inv}.  We also define its \textbf{Inv} to be the number of such pairs $(x,y)$ that form inversions, that is, $x<y$. 
\end{defn}

It remains to reinterpret $\pathdinv$, as in the following lemma.

\begin{defn}
        The \textbf{pathdinv} of a complete shrunken diagram is the number of pairs of boxes $x,y$ in the same row, with $x$ left of $y$, such that $y$ contains a label (possibly $0$), and $x$ is at most $n-k+1$ steps below the top entry of its column.  (Note that box $x$ does not have to contain a label; it may be an empty box.)
\end{defn}

\begin{lem}
\label{lem: pathdinv}
  We have $\pathdinv(\mathrm{shrink}(P))=\pathdinv(D)$ where $D$ is the path of $P$.
\end{lem}

\begin{proof}
  By Definition \ref{def:arm-leg}, a box $B$ that counts towards the $\pathdinv$ statistic of $D$ is a box above the path such that if $a$ is the arm of $B$ and $\ell$ is the leg of $B$ we have 
  \begin{equation}
  \ell\in \{as-1,as,as+1,as+2,\ldots,as+(s-1)\} \label{eq:legs}
  \end{equation}
  where $s=n-k+1$ is the slope of the parking function.

  Consider the column $C_1$ that $B$ is in, and the unique column $C_2$ of $D$ containing a label $y$ in the same row as $B$.  We also call $C_1,C_2$ the corresponding columns in $\mathrm{shrink}(P)$, and identify $y$ in $\mathrm{shrink}(P)$ as well.  
  
  Since we moved column $C_2$ down $(a+1)s$ steps vertically (relative to $C_1$) to make $\mathrm{shrink}(P)$, it follows that the new distance above the top of $C_1$ is $\ell-(a+1)s$, which by \eqref{eq:legs} is one of
  $$-(s+1),-s,-s+1,\ldots,-1.$$  Thus, the box $x$ in $C_1$ in $\mathrm{shrink}(P)$ that corresponds to $B$ in $D$ is one of the top $s+1$ boxes in column $C_1$ (possibly empty). 
\end{proof}

Putting together all of the above with Equation \ref{eq:tdinvs}, our goal is to show the theorem below.  We will use the following definition throughout.

\begin{defn}
    We call a column $C_i$ of a complete shrunken diagram {\bf basic}  if $h_i=1,h'_i=n-k+1$ and the labels in $C_i$ are exactly $0,1,\ldots,n-k$ from bottom to top. Note that the complement of a basic column is empty. 
\end{defn}

\begin{thm}
\label{thm: shrunken fixed point}
    For any generalized shrunken fixed point $P$, we have \begin{equation}
\Inv(P)+\pathdinv(P)=\maxInv(P) \label{eqn:equality}
    \end{equation}
\end{thm}

\begin{proof}
We prove the statement by changing $P$ in a series of reduction steps. At each step, we prove that the new diagram $P$ is again a generalized shrunken fixed point, and the quantity $\maxInv -\pathdinv-\Inv$ does not change.

As the base case, let $P_0$ be the unique generalized shrunken fixed point in which all columns are basic. Then  the complement of $P_0$ is empty, and $P_0$ is indeed a generalized shrunken fixed point. By Lemma \ref{lem: base case} below, \eqref{eqn:equality} holds for $P_0$.

Now let $P\neq P_0$, and let $C=C_i$ be the leftmost non-basic column. We claim that $h_i=1$. Indeed,  if $i=1$ then $h_i=1$ by assumption. If $i>1$ then $C_{i-1}$ is basic, so $h'_{i-1}=n-k+1$ and $h_i=1$.  Moreover, $h'_i\ge n-k+1$ by the definition of a shrunken diagram, and so $C$'s labels start at the bottom of the diagram, and $C$ contains at least one $0$. Since $C$ is not basic, either it does not contain all of $1,\ldots,n-k$, or it contains $1,\dots, n-k$ and more than one $0$.  We consider these two cases separately:

\textbf{Case 1.} Suppose $C$ does not contain all of $1,2,\ldots,n-k$, and let $m$ be the smallest positive integer that does not appear in $C$.  Let $P'$ be obtained from $P$ by removing the highest $0$ in column $C$ (which exists because $h'_i\ge n-k+1$), bumping down $1,2,\ldots,m-1$ each by one row, and inserting $m$ in the row that used to contain $m-1$ in column $C$ (see  Figure \ref{fig:m-bump}). Note that this step does not change $h_i$ or $h'_i$, so $P'$ is a complete shrunken diagram.

Let us prove that $P'$ is a generalized shrunken fixed point. Indeed, in the complement of $C$, we delete $m$ and do not shift any other labels. Therefore, the reading word for the complement of $P'$ is obtained from the one for $P$ by deleting $m$ and is strictly increasing. 

Finally, by Lemma \ref{lem:reduction-1} below, the quantity $\maxInv -\pathdinv-\Inv$ does not change.

\textbf{Case 2.} Suppose $C=C_i$ contains all labels $1,2,\ldots,n-k$ and more than one $0$.
In particular, $h'_i>n-k+1$ and $C$ is not the rightmost column (since $h_k'=n-k+1$). Also, $h_{i+1}>1$.

Let $P'$ be obtained from $P$ by removing the highest $0$ in column $C$, bumping down $1,2,\ldots,n-k$ each by one row in $C$, then placing an additional $0$ at the bottom of the column  $C_{i+1}$ to the right of $C$ (without shifting the rest of $C_{i+1})$. This decreases both $h'_i$ and $h_{i+1}$ by 1, so the equation $h'_{i}-(n-k)=h_{i+1}$ still holds and $P'$ is a complete shrunken diagram. See Figure \ref{fig:lower-C}.

The complement of $C$ is empty both in $P$ and $P'$, and the complement of $C_{i+1}$ is unchanged and not shifted. So the complements of $P$ and $P'$ are the same, and $P'$ is a generalized shrunken fixed point.

Finally, by Lemma \ref{lem:reduction-2}
below, the quantity $\maxInv -\pathdinv-\Inv$ does not change.
\end{proof}

\begin{lem}
\label{lem: base case}
  Define $P_0$ to be the unique generalized shrunken fixed point that consists of $k$ basic columns.  Then \eqref{eqn:equality} holds for $P_0$.
\end{lem}

\begin{proof}
   Note that any two columns have $n-k+1$ pathdinv's between them, $n-k$ (off-row) $\Inv$'s, and a total of $2n-2k+1$ $\maxInv$'s.  Thus the equality holds between each pair of columns, and hence holds globally for $P_0$.
\end{proof}

We now establish the first of two types of reduction steps.

\begin{figure}
    \centering
    \includegraphics{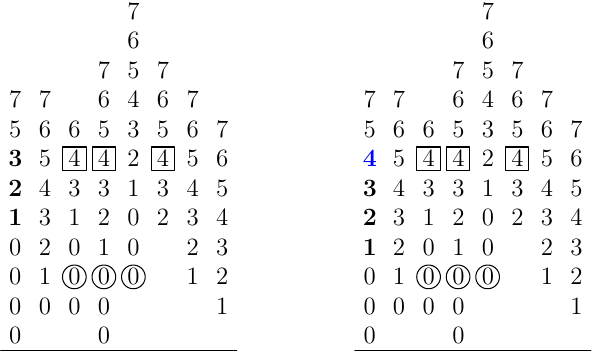}
    \caption{A parking function $P$ at left, and the parking function $P'$ at right formed by moving down the boldfaced elements $1,2,3$ in the first column and inserting a $4$ above them.  The $m$-losses and $0$-gains used in the proof of Lemma \ref{lem:reduction-1} are boxed and circled, respectively.  Notice that they alternate, when we read the columns from left to right, upwards within each column.
    \label{fig:m-bump}}
\end{figure}

\begin{lem}\label{lem:reduction-1}
    Let $P$ be a generalized shrunken fixed point, and let $C$ be the leftmost non-basic column of $P$.  Suppose $C$ does not contain all of $1,2,\ldots,n-k$, and let $m$ be the smallest positive integer that does not appear in $C$.  
    
    Let $P'$ be obtained from $P$ by removing the highest $0$ in column $C$, bumping down $1,2,\ldots,m-1$ each by one row, and inserting $m$ in the row that used to contain $m-1$ in column $C$.  Then $\Inv(P)=\Inv(P')$ (and $\pathdinv$ and $\maxInv$ remain unchanged as well).  
\end{lem}

\begin{proof}
Since the shape formed by the entries doesn't change, $\pathdinv$ and $\maxInv$ do not change from $P$ to $P'$. 

We now examine the changes in $\Inv$, which all involve boxes in column $C$.  First consider the $
\Inv$'s between column $C$ and boxes to its left; since the columns to the left are all filled with $0,1,2,\ldots,n-k$ and $C$ has height at least $n-k+1$, then for any entry $b$ in column $C$ of either $P$ or $P'$, the entries to its left in the same row are some entry $a\ge b$.  Thus, the $\Inv$'s to the \textit{left} with $C$ are precisely the off-row attacking pairs, and this does not change under switching to $P'$.

We now claim that, for any entry $y$ to the \textit{right} of column $C$, the number of $\Inv$'s involving $y$ and an entry of column $C$ remains unchanged except in the following two situations:

\begin{itemize}
    \item \textbf{$0$-gain:} $y$ is a $0$ in the row $r_0$ that contains the top $0$ in column $C$ of $P'$.
    \item \textbf{$m$-loss:} $y$ is an $m$ in the row $r_0+m$ that contains the new $m$ in column $C$ in $P'$.
\end{itemize}

Indeed, first suppose $y$ is in row $r_0$.  If $y=0$, we gain an off-row $\Inv$ between $y$ and the new position of the $1$ in column $C$ in $P'$, and we do not gain a horizontal $\Inv$.  If $y\neq 0$, it would not gain an off-row $\Inv$ and would retain its same-row $\Inv$.

Next suppose $y$ is in row $r_0+m$.  If $y=m$, we lose the horizontal $\Inv$ between the $m-1$ that was in row $r_0+m$ and $y$, and the off-row $\Inv$ value remains unchanged.  If $y\neq m$, both the horizontal and off-row $\Inv$ values are unchanged.

Finally, if $y$ is not in row $r_0$ or $r_0+m$, then either it is above row $r_0+m$ (in which case no $\Inv$'s change), below row $r_0$ (in which case no $\Inv$'s change), or between these two rows.  When it is between, the only time an $\Inv$ can change with $y$ is if $y=a$ where $a-1,a$ are the entries in its row and above it in column $C$ in $P$, and they change to $a,a+1$ in $P'$. Then $y$ loses a horizontal $\Inv$ with column $C$ and gains an off-row $\Inv$, and so its total $\Inv$ contribution remains unchanged.

We will now show the following two claims.  Consider the ordering on the letters of $P$ (and $P'$) formed by reading the columns from left to right, where we read bottom to top within each column.
\begin{enumerate}
    \item There is a $0$-gain between any two consecutive $m$-losses, and a $0$-gain before the first $m$-loss.
    \item There is an $m$-loss between any two consecutive $0$-gains, and an $m$-loss after the final $0$-gain.
\end{enumerate}

The two claims above will suffice, since then the $0$-gains and $m$-losses alternate and cancel with each other.  

To prove claim (1): let $C_1$ and $C_2$ be two consecutive columns containing $m$-losses (or, $C_1$ may be column $C$ itself, and $C_2$ the first column containing an $m$-loss).  Since the $m$ in $C_1$ is in row $r_0+m$, the highest possible height of column $C_1$ is if $m+1,\ldots,n-k$ all occur above the $m$, in which case its top entry has at most height $r_0+n-k$.  Thus the bottom of the next column is in or below row $r_0$.  If there is no $0$-gain between $C_1$ and $C_2$, the subsequent columns all remain sufficiently short that each next column's bottom entry is in or below row $r_0$.  Thus $C_2$'s bottom entry is in or below row $r_0$ as well, and since it has an $m$ in row $r_0+m$, the lowest row that can contain a nonzero entry in $C_2$ is row $r_0+1$ (if we put all of $1,2,\ldots,m-1$ below the $m$ in $C_2$).  Thus $C_2$ contains a $0$ in row $r_0$, so there is indeed a $0$-gain before the $m$-loss in $C_2$ (and after $C_1$).

To prove claim (2): let $C_1$ and $C_2$ now be two consecutive columns to the right of $C$ containing $0$-gains.  That is, each has a $0$ in row $r_0$.  Notice that since $m$ is in the complement of column $C$ in $P$, it is not in the complement of any other column, so in particular $m$ appears in all columns to the right of $C$.  

Consider the height of $m$ in column $C_1$.  If it is in row $r_0+m$, we have found an $m$-loss between the two $0$-gains and we are done.  Now suppose it occurs below row $r_0+m$.  Consider the letters from the complement $\overline{C}_1$ plus those in $C_1$ that are less than $m$.  Then these letters stretch at least as low as row $r_0$, but there is a zero in row $r_0$ so some letter $x<m$ from $\overline{C}_1$ must be in row $r_0$ in the complement.  Since $m$ occurs in row $r_0+1$ in the complement of $P$ in column $C$, this letter $x<m$ occurs after $m$ in reading order in the complement, a contradiction.  Thus the $m$ in $C_1$ occurs above row $r_0+m$.

We now show that this forces the height of column $C_1$ to reach at least row $r_0+(n-k)+1$.  To do so, it suffices to show that the letters $m+1,m+2,\ldots,n-k$ must all appear in $C_1$ above the $m$.  Since $m$ is strictly above row $r_0+m$, the letters $1,2,\ldots,m-1$ in both $C_1$ and the complement $\overline{C}_1$ occur strictly above row $r_0+1$ as well.  Thus if any letter among $m+1,m+2,\ldots,n-k$ is missing in $C_1$, then the smallest such, say $s$, is in $\overline{C}_1$ and occurs in or above row $r_0$ in the complement $\overline{P}$.  But then $s$ occurs before $m$ in reading order in $\overline{P}$ and $s>m$, a contradiction.  Hence $C_1$ reaches row height at least $r_0+(n-k)+1$.

Then, the column just to the right of $C_1$ has its bottom (possibly $0$) entry above row $r_0$, and the same analysis shows that if we do not have an $m$-loss in this column, the subsequent bottom entries will continue to lie above row $r_0$.  Thus if there is no $m$-loss between $C_1$ and $C_2$, the bottom entry of $C_2$ is also above $r_0$, a contradiction since it contains a $0$-gain.  This completes the proof.
\end{proof}

We now establish the second reduction step.

\begin{lem}
\label{lem:reduction-2}
    Let $P$ be a generalized shrunken fixed point, and let $C$ be the leftmost non-basic column of $P$.  Suppose $C$'s entries are $0,0,\ldots,0,1,2,\ldots,n-k$ from bottom to top, so in particular there is more than one $0$.
    
    Let $P'$ be obtained from $P$ by removing the highest $0$ in column $C$, bumping down $1,2,\ldots,n-k$ each by one row in $C$, then placing an additional $0$ at the bottom of the column to the right of $C$.  Then $$\maxInv(P)-\pathdinv(P)-\Inv(P)=\maxInv(P')-\pathdinv(P')-\Inv(P').$$
\end{lem}

\begin{figure}
    \centering
    \includegraphics{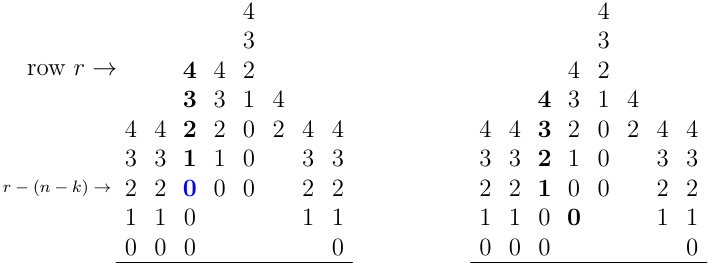}
    \caption{Lowering the column $C$ shown with boldface entries one step, to form  $P'$ at right.}
    \label{fig:lower-C}
\end{figure}

\begin{proof}
    Let $r$ be the row of the box in column $C$ of $P$ that is labeled by $n-k$.  Then the difference $\pathdinv(P')-\pathdinv(P)$ consists of the following changes:
    \begin{itemize}
        \item \textbf{Upper pathdinv losses:} There is one pathdinv lost for each entry in row $r$ to the right of column $C$ (from removing the top box from column $C$). 
        \item \textbf{Lower pathdinv gains:} There is one pathdinv gained for each entry in row $r-(n-k)-2$ to the right of column $C$ (since column $C$ lowered its height by $1$, so the $x$ in the pathdinv can now be one row lower in $C$ than before),  
        \item \textbf{Extra pathdinvs} with the new $0$ entry:  We gain one more pathdinv between the new $0$ entry to the right of column $C$ and every entry to its left.  (Note that every box to its left is filled with an entry that is close enough to the top of their columns to count towards pathdinv.) 
    \end{itemize}

    The difference $\maxInv(P')-\maxInv(P)$ consists of the following:
    \begin{itemize}
        \item \textbf{Upper maxInv losses, row $r$:}  There is one maxInv lost for each entry in row $r$ to the right of column $C$.
        \item \textbf{Upper maxInv losses, row $r-1$:} There is one maxInv lost for each entry in row $r-1$ to the right of column $C$.
  
        \item \textbf{Rightwards off-row maxInv gains} with the new $0$ entry:  We gain one more maxInv between the new $0$ entry to the right of column $C$ and every entry to its right in the row below it, row $r-(n-k)-2$.  
        \item \textbf{Leftwards off-row maxInv gains} with the new $0$ entry: We gain one more maxInv between the new $0$ entry and every entry to its left in the row $r-(n-k)$ above it.
        \item \textbf{Rightwards horizontal maxInv gains} with the new $0$ entry:  We gain one more maxInv between the new $0$ entry to the right of column $C$ and every entry to its right in row, $r-(n-k)+1$.
        
        \item \textbf{Leftwards horizontal maxInv gains} with the new $0$ entry:  We gain one more maxInv between the new $0$ entry to the right of column $C$ and every entry to its left in row, $r-(n-k)+1$.
        
    \end{itemize}

    Clearly the upper pathdinv losses cancel with the upper maxInv losses in row $r$, and the lower pathdinv gains cancel with the rightwards off-row maxInv gains.  The extra pathdinvs also cancel with the leftwards horizontal maxInv gains.  It follows that we only need to match the remaining changes in maxInv with the changes in $\Inv$, where the remaining maxInv changes are:
    \begin{itemize}
        \item \textbf{Upper maxInv losses, row $r-1$}.
        \item \textbf{Leftwards off-row maxInv gains} with the new $0$ entry.
        \item \textbf{Rightwards horizontal maxInv gains} with the new $0$ entry.
    \end{itemize}

    Notice that the leftwards off-row maxInv gains all involve boxes to the left filled with something greater than $0$, which are precisely the new $\Inv$'s formed to the left.   Thus we can now focus on the $\Inv$'s only involving boxes weakly to the right of column $C$ (not counting the new $\Inv$ between the new $0$ and the $1$ in column $C$, which we have already counted), and show that they cancel with:
    \begin{itemize}
        \item \textbf{Upper maxInv losses, row $r-1$}.
        \item \textbf{Rightwards horizontal maxInv gains} with the new $0$ entry.
    \end{itemize}

    Consider any numbered entry $y$ to the right of column $C$, occurring in some row other than row $r-1$ or row $r-(n-k)-1$.  We claim that the number of Inv's with $y$ as the right hand entry is unchanged.  Indeed, if $y$ is in row $r$, there were no $\Inv$'s formed with it from column $C$, because the entry $n-k$ in that row from column $C$ is as large as possible, and there are no entries above.  Thus if $y$ is in row $r$ or above, its $\Inv$ contribution does not change. 

    If $y$ is in a row lower than $r-(n-k)-1$, it also does not gain or lose $\Inv$'s because it only sees $0$'s to the left of it in column $C$, and the new $0$ in row $r-(n-k)-1$ in $P'$ does not form an $\Inv$ with $y$ since it is smaller than, and above, $y$.

    If $y$ is in a row between $r-(n-k)$ and $r-2$ inclusive, its $\Inv$'s do not change with any column besides column $C$ since those entries are not altered.  Comparing $y$ to column $C$, let $a-1$ and $a$ be the entries in column $C$ in $P$ that are in the same row, and in the row above, $y$ respectively.  Then in $P'$, these entries change to $a$ and $a+1$ respectively.  If $y\le a-1$, the upper entry forms an $\Inv$ and the lower does not in both $P$ and $P'$.  If $y\ge a+1$, the lower entry forms an $\Inv$ in both and the upper does not.  If $y=a$, it changes from being one horizontal $\Inv$ to one off-row $\Inv$.  In all cases, the $\Inv$ contribution of $y$ is unchanged.

    Now, consider an entry $y$ to the right of column $C$ in row $r-(n-k)-1$ in $P$ (where $y$ is also not the box that becomes the new $0$).  If $y=0$, then it did not form an $\Inv$ with column $C$ in $P$, but it does form an $\Inv$ with the new position of the $1$ in column $C$ in $P'$.  If $y>0$, it forms the same number of $\Inv$'s with column $C$ as before, but it forms a unique new $\Inv$ with the new $0$ in the column to the right of $C$.  Either way it contributes one new $\Inv$ in $P'$.  These extra $\Inv$'s cancel with the rightwards horizontal maxInv gains with the new $0$ entry.

    Finally, consider an entry $y$ to the right of column $C$ in row $r-1$.  If $y$ is less than $n-k$, then we lose an $\Inv$ from the $n-k$ in column $C$ moving down to row $r-1$.  If $y$ equals $n-k$, then we lose the $\Inv$ that we had from the $n-k-1$ in column $C$.  Either way we lose exactly one $\Inv$ for each such $y$, and so these cancel precisely with the upper maxInv losses in row $r-1$.  
\end{proof}

\subsection{Relation to LLT and Kazhdan-Lusztig polynomials}

In this subsection, we give a different interpretation of Theorem \ref{thm: perping big main} following \cite{HHLRU}.

Let $\underline{\mu}=(\mu^{(0)},\ldots,\mu^{(k-1)})$ be a $k$-tuple of Young diagrams, and let $\underline{c}=(c_0,\ldots,c_{k-1})$ be a sequence of integers called \emph{shifts}. Given a box $x\in \mu^{(i)}$, we define its \textbf{position content} $c(x)$ to be the number of squares to its left minus the number of squares below it.  The {\bf adjusted position content} of $x$ is
$$
\widetilde{c}(x)=kc(x)+c_i.
$$
Let $\underline{T} = (T_0,\dots, T_{k-1})$ be a tuple of semistandard tableaux such that $T_i$ has shape $\mu^{(i)}$.  A pair of boxes $x,y$ is called an \textbf{LLT inversion pair} if
$$
x<y\ \mathrm{and}\ 0<\widetilde{c}(x)-\widetilde{c}(y)<k.
$$
We denote by $\inv_{\mathrm{LLT}}(\underline{T})$ the number of LLT inversion pairs in $\underline{T}$.

\begin{lem}
\label{lem: LLT inversions}
Consider a $(K,k)$ Dyck path $D$ with a choice $(\underline{b},\underline{s})$ of big and small labels, as in Theorem \ref{thm: perping big main}. Let $\underline{\mu}=(1^{b_k},\ldots,1^{b_1})$ and $c_i=i-k(h_{k-i}+s_{k-i})$ where $h_i$ is given by \eqref{eq: hi shrink}.
Then:
\begin{itemize}
\item There is a bijection between $\WPF(D,\underline{b})$ and the set of semistandard tableaux of shape $\underline{\mu}$.
\item Under this bijection, we have
 $\tdinv_{\bg}(\pi_{\bg})=\Inv_{\mathrm{LLT}}(\underline{T})$.
 \end{itemize}
\end{lem}

\begin{proof}
The first part is clear, as we match a box in column $C_i$ with the corresponding box in the diagram $\mu^{(k-i)}$ (we recall that SSYT must be strictly increasing in columns and weakly increasing in rows). 

For the second part, we shrink $\pi_{\bg}$ and use the equation $\tdinv_{\bg}(\pi_{\bg})=\Inv(\mathrm{shrink}(\pi_{\bg}))$ from Proposition \ref{prop: shrink}.
To conclude the proof, we need to prove that $\Inv(\mathrm{shrink}(\pi_{\bg}))=\Inv_{\mathrm{LLT}}(\underline{T})$. Indeed, suppose that a box $x$ in column $C_i$ in $\mathrm{shrink}(\pi_{\bg})$ has height $t_x$. Since there are $(h_i-1)$ empty spaces and $s_i$ zeros in column $C_i$ below $x$, the position content of the corresponding box in $\mu^{(k-i)}$ equals $c(x)=-(t_x-h_i-s_i)$, so the adjusted content equals
$$
\widetilde{c}(x)=-k(t_x-h_i-s_i)+c_{k-i}=-k(t_x-h_i-s_i)+(k-i)-k(h_i+s_i)=k-i-kt_x.
$$
Now suppose that $y$ is in column $j$ at height $t_y$. The inequality
$0<\widetilde{c}(x)-\widetilde{c}(y)<k$ can be rewritten as
$0<(j-i)+k(t_y-t_x)<k$, which occurs if either $i<j,t_x=t_y$ or $i>j,t_y=t_x+1$.  Thus $x,y$ form an LLT inversion pair if and only if they contribute to $\tdinv$.
\end{proof}

Given a tuple of diagrams $\underline{\mu}=(\mu^{(0)},\ldots,\mu^{(k-1)})$ and a collection of shifts $\underline{c}=(c_0,\ldots,c_{k-1})$, \cite{HHLRU} associate a partition $\mu=\mu(\underline{\mu},\underline{c})$ such that the $k$-quotient of $\mu$ is $\underline{\mu}$ and the $k$-core of $\mu$ is determined by $\underline{c}$.
We refer to \cite{HHLRU} for more details.
Furthermore, for such $\mu$ \cite{LLT} define an LLT polynomial $G_{\mu}(X;q)$. The following theorem can be used as an alternative definition.

\begin{thm}\cite[Corollary 5.2.4]{HHLRU}
One has
$$
q^{e}G_{\mu(\underline{\mu},\underline{c})}[X;q^{-1}]=\sum_{\underline{T}\in \mathrm{SSYT}(\underline{\mu})}q^{\Inv_{\mathrm{LLT}}(\underline{T})}x^{\underline{T}}
$$
for some exponent $e$.
\end{thm}

\begin{cor}
For $\underline{\mu}$, $\underline{c}$ as in Lemma \ref{lem: LLT inversions} and we have
$$
q^{e}G_{\mu(\underline{\mu},\underline{c})}[X;q^{-1}]=f_{D,\underline{b}}[X;q].
$$
\end{cor}

Finally, we can make a connection to parabolic Kazhdan-Lusztig polynomials.

\begin{thm}\cite[Proposition 5.3.1]{HHLRU}
Suppose that $\mu=\mu(\underline{\mu},\underline{c})$ and $\nu$ is the $k$-core of $\mu$. Then for all $\lambda$ one has
$$
q^{\mathrm{smin}(\mu)}\langle G_{\mu}[X,q^2],s_{\lambda}\rangle = P^{-}_{\mu+\rho,\nu+k\lambda+\rho}(q)
$$
where the right hand side is the parabolic Kazhdan-Lusztig polynomial, and $\mathrm{smin}$ is a certain combinatorial statistic.
\end{thm}

As a consequence, we can interpret the pairing 
$$
\langle f_{D,\underline{b}}[X,q^{-2}],s_{(k-1)^{(n-k)}}\rangle
$$
from Theorem \ref{thm: perping big main} as a parabolic Kazhdan-Lusztig polynomial, up to monomial factors $q^{e}$ and $q^{\mathrm{smin}(\mu)}$. Theorem \ref{thm: perping big main} then implies that this parabolic Kazhdan-Lusztig polynomial is in fact a monomial. 
This can be proved directly as follows.

\begin{lem}
For $\underline{\mu},\underline{c}$ as in Lemma \ref{lem: LLT inversions}, and $\lambda=s_{(k-1)^{(n-k)}}$ the parabolic Kazhdan-Lusztig polynomial $P^{-}_{\mu+\rho,\nu+k\lambda+\rho}(q)$ is a monomial in $q$.
\end{lem}

\begin{proof}
It is a deep result of \cite{LLT} that  all coefficients of the  parabolic Kazhdan-Lusztig polynomial  are nonnegative, thus it is sufficient to prove that the value of this polynomial at $q=1$ equals 1. 

On the other hand, by the above at $q=1$ we get
$$
\langle f_{D,\underline{b}}[X;1],s_{(k-1)^{(n-k)}}\rangle=1.
$$
This follows from \cite[Lemma 3.5]{GG}, since at $q=1$ the polynomial $f_{D,\underline{b}}$ specializes to the product $e_{b_1}\cdots e_{b_k}$.
\end{proof}

It would be interesting to  compute the power of $q$ in Theorem \ref{thm: perping big main} directly by this method. This would give an alternate proof of this theorem.

\section{Further directions}

In this section, we list some further directions stemming from our work.

    \subsection*{Connection to other formulas} Another (conjectural) formula for the Delta Conjecture symmetric function in terms of a skewing operation has been previously discovered by Bergeron~\cite{BergeronConj}. He conjectured that
    \begin{equation}\label{eq:Bergeron}
    \Delta'_{e_{k-1}}e_n = (e_{n-k}^\perp \otimes \mathrm{Id})\mathcal{E}_n,
    \end{equation}
    where $\mathcal{E}_n$ is (the character of) the space of $m$-variate diagonal harmonic polynomials for sufficiently large $m$, thought of as a $GL_m\times S_n$-module. Here, $e_{n-k}^\perp$ is an operator on the $GL_m$ character of the space, and the skewed $GL_m$ characters yield the $q,t$ coefficients in the expansion into Schur functions. On the other hand, our skewing formula Theorem~\ref{thm:IntroSkewing} for $\Delta'_{e_{k-1}}e_n$ starts with a larger degree symmetric function and uses $s_{(k-1)^{n-k}}$ to reduce it to the correct degree $n$ symmetric function. Thus, it would be interesting to investigate the following.
    
    \begin{question}
        Is there a direct connection between these two skewing formulas? Is there an analogue of \eqref{eq:Bergeron} for the symmetric function $E_{K,k}\cdot 1$?
    \end{question}
    
\subsection*{Approach to the Valley formula}     In this article, we prove that the Rational Shuffle Theorem implies the Rise Delta Theorem. As we mentioned in the introduction, there is another version of the Delta Conjecture due to Haglund, Remmel, and Wilson~\cite{HRW}, called the Valley Delta Conjecture, that remains open. 

\begin{question}
    Can Theorem~\ref{thm:IntroSkewing} be used to prove the Valley Delta Conjecture?
\end{question} 

We expect that the basic structure of our proof still applies to this setting but that the combinatorics involved would be quite different.

\subsection*{Extensions of the skewing formula}
    
     There are many extensions of the skewing formula Theorem~\ref{thm:IntroSkewing} that one could explore for $\Delta'$ and the related $\Theta$ operators of D'Adderio, Iraci, and Vanden Wyngaerd~\cite{Theta}. For example:
     \begin{question}     
 Can $\Delta_{h_\ell}\Delta'_{e_{k-1}} e_n$ or  $\Delta'_{s_\lambda}s_\mu$ be expressed as a skewing operator applied to some evaluation of an Elliptic Hall Algebra operator on $1$?
    \end{question}
    
\bibliographystyle{plain}
\bibliography{refs.bib}

\end{document}